\documentclass[12pt,reqno]{amsart}
\usepackage{ifthen,amsfonts,amsmath,amssymb,epsfig,color,graphicx,amsthm,mathrsfs,amsopn}
\usepackage[a4paper,left=35mm,right=35mm,top=30mm,bottom=30mm,marginpar=25mm]{geometry}

\usepackage{latexsym}

\usepackage{microtype}
\usepackage[colorlinks=true, citecolor=blue]{hyperref}

\makeatletter

\numberwithin{equation}{section}

\newcommand{\Rom}[1]{\uppercase\expandafter{\romannumeral #1}}
\newcommand{\Fc}{\mathbf{F}^c}
\newcommand{\Si}{\Sigma}
\newcommand{\mS}{\mathcal{S}}
\newcommand{\mR}{\mathcal{R}}

\newcommand{\an}{{\rm An}}
\newcommand{\eps}{{\varepsilon}}

\newcommand{\Lip}{{\text {Lip}}}
\newcommand{\diam}{{\text {diam}}}

\newcommand{\R}{{\rm R}}

\def\RR{{\mathbb{R}}}

\def\NN{{\mathbb{N}}}

\newcommand{\e}{{\text {e}}}

\newcommand{\cC}{{\mathcal{C}}}
\newcommand{\cD}{{d}}

\newcommand{\cF}{{\mathcal{F}^c}}

\newcommand{\cH}{{\mathcal{H}}}

\newcommand{\cO}{{\mathcal{O}}}

\newcommand{\cV}{{\mathcal{V}}}

\newcommand\B{{\mathcal{B}}}

\renewcommand\d{{\rm d}\,}
\def\rn#1{{\bf R}^{#1}}

\newcommand\haus{\mathcal{H}}
\newcommand\Z{{\mathbb Z}}
\newcommand\N{{\mathbb N}}

\newcommand\res{\mathop{\hbox{\vrule height 7pt width .5pt depth 0pt
\vrule height .5pt width 6pt depth 0pt}}\nolimits}
\newcommand\supp{{\rm supp}\,}  
\newcommand\Inj{{\textrm{Inj}\,}}

\newcommand\FF{{\mathbf{F}}}

\newcommand\de{{\rm Deg}\,}
\newcommand\aan{{\mathcal{A}n}}
\newcommand\An{{\mathcal{AN}}}

\newcommand\ind{{\bf 1}}

\def\Id{\mathrm{Id}}

\newtheorem{Theorem}{Theorem}[section]
\newtheorem{Lemma}[Theorem]{Lemma}
\newtheorem{Proposition}[Theorem]{Proposition}

\newtheorem{Definition}[Theorem]{Definition}

\theoremstyle{remark}
\newtheorem{Remark}[Theorem]{Remark}

\numberwithin{equation}{section}

\newcommand{\norm}[1]{\left\lVert#1\right\rVert}

\def\H{\mathcal H}

\def\R{\mathbb R}
\def\N{\mathbb N}
\def\Z{\mathbb Z}
\def\e{\varepsilon}

\def\vphi{\varphi}

\def\Om{\Omega}
\def\de{\delta}
\def\Id{{\rm Id}}

\def\spt{{\rm spt}}

\def\pa{\partial}

\def\rn#1{{\mathbb R}^{#1}}

\renewcommand{\d}{\mathrm{d}}

\def\Lip{{\rm Lip}\,}

\def\C{\mathbf{C}}

\def\B{\mathcal{B}}

\def\Id{\mathrm{Id}}

\def\Inj{\mathrm{Inj}}

\DeclareMathOperator{\Div}{div}

\begin{document}

\title[The anisotropic Min-Max theory]
{The anisotropic Min-Max theory: Existence of anisotropic minimal and CMC surfaces}

\author[G. De Philippis]{Guido De Philippis}
\address{G. De Philippis: Courant Institute of Mathematical Sciences, New York University, 251 Mercer St., New York, NY 10012, USA.}
\email{guido@cims.nyu.edu}

\author[A. De Rosa]{Antonio De Rosa}
\address{A. De Rosa: Department of Mathematics, University of Maryland, 4176 Campus Dr, College Park, MD 20742, USA}
\email{anderosa@umd.edu}

\begin{abstract}
We prove the existence of nontrivial closed surfaces with constant anisotropic mean curvature with respect to elliptic integrands in closed smooth $3$--dimensional Riemannian manifolds. The constructed min-max surfaces are smooth with at most one singular point. The constant anisotropic mean curvature can be fixed to be any real number. In particular, we  partially solve a conjecture of Allard \cite{Allard1983} in dimension $3$.
\end{abstract}

\maketitle
 
\section{Introduction}
The min-max theory has been extremely successful in finding critical points of the area functional, commonly referred to as minimal surfaces. Almgren \cite{Alm0} started a  monumental program to develop a variational theory for minimal surfaces of arbitrary dimension and codimension using geometric measure theory. In particular, he proved the existence of a nontrivial weak solution as a stationary integral varifold \cite{Alm}. The regularity of the solution in codimension one was established  by Pitts \cite{P} up to dimension $6$ of the ambient manifold, relying on the curvature estimates for stable minimal surfaces proved by Schoen-Simon-Yau \cite{SSY}. Then it has been extended to every dimension by Schoen-Simon \cite{SS} (always in codimension one). Thereafter, Colding-De Lellis \cite{CD} and De Lellis-Tasnady \cite{DLTas} have proved a similar construction using smooth sweepouts, building on ideas of Simon-Smith \cite{Sm}. Very recently, the zero mean curvature case has been generalized by Zhou-Zhu \cite{ZZ} to the existence of constant mean curvature surfaces in closed Riemannian manifolds. 

In spite of the aforementioned vast literature for the area functional, nothing is known concerning the existence of anisotropic minimal or constant anisotropic mean curvature surfaces.
Allard \cite{Allard1983} proved an analogue of the curvature estimates of Schoen-Simon \cite{SS} and conjectured the existence of anisotropic minimal surfaces in closed Riemannian manifolds \cite[Page 288]{Allard1983}. However he observed that: \emph{``there remains a considerable amount of works to do before this becomes feasible for general integrands''}\cite[Page 288]{Allard1983}.

In this paper we partially solve this problem in dimension $3$, proving the following first main result:
\begin{Theorem}\label{t:existence}
 Let $M$ be a $3$-dimensional $C^4$ closed Riemannian manifold and $F$ be a $C^3$ elliptic integrand. Then there is a nontrivial surface $\Sigma\subset M$ without boundary which is a critical point with respect to $F$. Moreover there exists at most one singular point $p\in M$ for $\Sigma$, i.e. $\Sigma$ is $C^2$ embedded away from $p$. Furthermore one of the following properties hold:
\begin{itemize}
\item[(a)] there exists $R>0$ such that $\Sigma$ is smooth stable in $B_R(x)$ for every $x\in  M$;
\item[(b)] $\Sigma$ is stable in $M\setminus \{p\}$.
\end{itemize}

\end{Theorem}

Theorem \ref{t:existence} is obtained from the following second main result:
\begin{Theorem}\label{t:existenceCMC}
 Let $M$ be a $3$-dimensional $C^4$ closed Riemannian manifold, $F$ be a $C^3$ elliptic integrand and $c\in \R\setminus \{0\}$. Then there is a nontrivial surface $\Sigma\subset M$ without boundary which has constant anisotropic mean curvature $c$ with respect to $F$. Moreover there exists at most one singular point $p\in M$ for $\Sigma$, i.e. $\Sigma$ is $C^2$ almost embedded away from $p$. Furthermore one of the following properties hold:
\begin{itemize}
\item[(a)] there exists $R>0$ such that $\Sigma$ is smooth stable in $B_R(x)$ for every $x\in  M$;
\item[(b)] $\Sigma$ is stable in $M\setminus \{p\}$.
\end{itemize}
\end{Theorem}

The assumptions of a $C^4$ Riemannian manifold and a $C^3$ Lagrangian are used to construct a surface $\Sigma$ of class $C^2$. An higher regularity on $M$ and $F$ provides a better regularity also for $\Sigma$. In this paper, if not further specified, smooth will always refer to the above regularities. 

Theorem \ref{t:existence} is the $3$-dimensional anisotropic counterpart of the celebrated existence result for the area functional proved by Pitts in his groundbreaking monograph \cite{P}.
The isotropic version of Theorem \ref{t:existenceCMC} has been recently proved by Zhou-Zhu in \cite{ZZ}.

As previously remarked, Allard has conjectured in \cite{Allard1983} the possibility to run the Pitts min-max construction scheme \cite{P} in the anisotropic setting. Theorem \ref{t:existence} positively answers to this conjecture for $3$-dimensional manifolds.

\medskip

{\bf Strategy of the proof.} The proof follows the scheme developed by Almgren-Pitts \cite{P} for the existence of isotropic minimal surfaces. The scheme consists of two parts: the existence and the regularity part. The existence of an anisotropic stationary rectifiable varifold can be proved with a similar strategy to the one used by Almgren-Pitts  \cite{P}, replacing Allard rectifiability theorem \cite{All} with its anisotropic counterpart proved by the authors \cite{DDG2}. To this aim we need to employ a density lower bound, proved by Allard in \cite{Allard1983}. On the other hand the regularity part has several obstructions. The main issue is in showing that two consecutive replacements glue smoothly. Indeed, due to the lack of monotonicity formula \cite{Allardratio}, it is not clear a priori if upper density estimates hold at the points of the interface of two consecutive replacements. Moreover, the lack of monotonicity formula does not allow to directly prove that the blowups at these points are planes by simple mass comparison at different scales. To solve these issues, our strategy is first to focus on the existence of constant non-zero anisotropic mean curvature (CMC) surfaces. Borrowing ideas from \cite{ZZ}, in the CMC setting we can construct multiplicity one replacements,  containing just a 1-dimensional set of touching points with multiplicity 2. This allows to split the proof of the smooth gluing at the interface in two cases:
\begin{itemize}
\item For multiplicity one points we can refine the construction of the second replacement, showing that its approximating sequence is regular up to the interface. Using its stability and the work of the authors \cite{DDH}, we get uniform boundary curvature estimates, which allow to pass to the limit the sequence and to obtain a replacement that is smooth (and stable) across the interface. For the boundary curvature estimate we deeply rely on the surface being two-dimensional.
\item For the isolated points of multiplicity two, we can show upper density estimates adapting \cite{DDH}. The blowups are then proved to be planes with multiplicity two, exploiting the regularity in the points of multiplicity one and the maximum principle. Then by a graphicality argument we conclude that the two replacements glue smoothly. 
\end{itemize}
This allows to show that the constructed anisotropic CMC surface is smooth and locally stable away from finitely many points: the centers of the balls of the previous argument. Again due to the lack of monotonicity formula we cannot remove these singularities just using the stationarity and stability. However, refining Almgren-Pitts combinatorial lemma, by compactness we are able to remove all these isolated singularities, except one. This last singularity is the one accounting for the index of the constructed surface and cannot be removed by further exploiting Almgren-Pitts combinatorial lemma. However, we would expect that this singularity is removable, too, by means of the PDE and of the stability inequality outside of the singularity. The last step is to apply this construction to obtain a sequence of CMC surfaces $\Sigma_k$ with anisotropic mean curvature $1/k$, smooth away from at most one point $x_k$. A simple compactness theorem will guarantee the convergence of $\Sigma_k$ to an anisotropic minimal surface, smooth outside of at most one point.

\subsection*{Acknowledgments} Guido De Philippis has been partially supported by the NSF grant DMS 2055686 and by the Simons Foundation. Antonio De Rosa has been partially supported by the NSF DMS Grant No.~1906451, the NSF DMS Grant No.~2112311, and the NSF DMS CAREER Award No.~2143124.

\section{Preliminaries}\label{s:prel}
In this paper, unless otherwise specified, $M$ will denote a $3$-dimensional smooth Riemannian manifold without boundary.
\subsection{Notation}\label{ss:notat}
We first recall some classical notation: $\Inj(M)$ is the injectivity radius of $M$, $\diam (S)$ is the diameter of $S\subset M$ and $\d(S_1, S_2):=\inf_{x\in S_1, y\in S_2} \d (x,y)$  for every $S_1, S_2\subset M$.

We denote with $B_r(x)$ and $\overline{ B_r(x)}$ respectively the open and closed balls centered at $x$ of radius $r$ with respect to the metric of the Riemannian manifold $M$. $\B_r(x)$ denotes instead the ball of radius $r$ and centered in $x$ in $\RR^3$. If the center is $0$, we simply write $\B_r$.
We set 
$$\an(x,s,r):=B_r(x)\setminus \overline{B_s(x)}, \quad \aan (x,s,r):=\B_r(x)\setminus \overline{\B_s(x)}$$
 to be the open annulus centered at $x$ of radii $s$ and $r$. Furthermore we will define $$\An_r(x):=\{\an(x,s,t) \, : \, 0<s<t<r\}, \quad \mbox{and} \quad \An_\infty(x):=\{\an(x,s,t) \, : \, 0<s<t\}.$$
 Given an annulus $An$, we denote with $\partial_+ An$ the largest connected component of $\partial An$.

We denote the space of smooth (respectively compactly supported) vector fields on $M$ as $\mathcal X(M)$ (respectively $\mathcal X_c(M)$).

\subsection{Measures, rectifiable sets and varifolds}
Given a locally compact metric space $Y$, we denote by $\mathcal M_+(Y)$ the set of positive Radon measures in $Y$. For a  Borel set \(A\subset Y\),  \(\mu\res A\) is the  restriction of \(\mu\) to \(A\), i.e. the measure defined by \([\mu\res A](E)=\mu(A\cap E)\) for every Borel set $E\subset Y$. 

For every $n \in \N$, we denote by  $\haus^n$   the  $n$-dimensional Hausdorff measure.
A  set \(K\) is said to be \(2\)-rectifiable if it can be covered, 
up to an \(\H^2\)-negligible set, by countably many \(C^1\) $2$-dimensional 
submanifolds.   Given a $2$-rectifiable set $K$, we denote $T_xK$ the approximate tangent space of $K$ at $x$, which exists for $\H^2$-almost every point $x \in K$, \cite[Chapter 3]{Si}. 

For  \(\mu\in \mathcal M_+(M) \) we consider its lower and upper  \(2\)-dimensional densities at \(x\):
\[
\theta_*(x,\mu):=\liminf _{r\to 0} \frac{\mu(B_{r}(x))}{ \pi r^2}, \qquad \theta^{*}(x,\mu):=\limsup_{r\to 0} \frac{\mu(B_{r}(x))}{ \pi r^2},
\]
In case these  two limits are equal, we denote by \(\theta(x,\mu)\) their common value. Note that, if $\mu= \theta\H^2 \res K$ where \(K\) is $2$-rectifiable, then \(\theta(x)=\theta_*(x,\mu)=\theta^{*}(x,\mu)\) for \(\mu\)-a.e. \(x\), see~\cite[Chapter 3]{Si}, and $\mu$ is called a $2$-rectifiable measure.

For every open subset $\Omega \subset M$, we denote
 $$G_2(\Omega):=\{(x, S) \, : \, x\in \Omega, \, S \mbox{ is a $2$-dimensional linear subspace of } T_xM\},$$
 and for every open subset $U\subset \R^3$
 $$G_2(U):=\{(x, S) \, : \, x\in U, \, S \mbox{ is a $2$-dimensional linear subspace of } \R^3\},$$
We denote with $\cV(M):=\mathcal M_+(G_2(M))$ (resp. $\cV(U):=\mathcal M_+(G_2(U))$) the space of the $2$-varifolds on $M$ (resp. on $U$). Given a diffeomorphism $\psi \in C^1(M,M)$, we define the push-forward of $V\in\mathbb \cV(M)$ with respect to $\psi$ as the varifold $\psi_\#V\in \cV(M)$ such that
$$\int_{G_2(M)}\Phi(x,S)d(\psi_\#V)(x,S)=\int_{G_2(M)}\Phi(\psi(x),d_x\psi(S))J\psi(x,S) dV(x,S),$$
for every $\Phi\in C^0_c(G_2(M))$. Here $d_x\psi(S)$ is the image of $S$ under the linear map $d_x\psi(x)$ and $J\psi(x,S)$ denotes the $2$-Jacobian determinant  (i.e. area element) of the differential $d_x\psi$ restricted to the $2$-plane $S$, 
 see \cite[Chapter 8]{Si}.

To a varifold \( V\in \cV(M) \) (resp. $V\in\cV(U)$), we associate the measure  $\|V\|\in \mathcal M_+(M)$  (resp. $\|V\|\in \mathcal M_+(U)$) defined by 
\[
\|V\|(A)=V(G(A))\qquad\textrm{for every open \(A\subset M\) (resp. \(A\subset U\)).}
\]
We will also use the notation
\[
\theta_*(x,V)=\theta_*(x,\|V\|)\qquad\textrm{and}\qquad\theta^{*}(x,V)=\theta^{*}(x,\|V\|)
\] 
for the upper and lower densities of \(\|V\|\). In case $\theta_*(x,V)=\theta^{*}(x,V)$, we denote the common value $\theta(x,V)$ and it will be referred to as density of $V$ at $x$.

 For every $x\in M$ and $r< \Inj (M )$, we define $T_x^r : z\in \B_1\to \exp_x(rz) \in B_r(x)$, where
$\exp_x$ denotes the exponential map at the point $x$. Then we define the function  \(\eta_{x,r}(y):=(T_x^r)^{-1}(y)\). We denote
\begin{equation*}\label{blowupvarifolds}
V_{x,r}:=(\eta_{x,r})_\#(V\res B_r(x)) \in \cV(\B_1), 
\end{equation*}
and we observe that, if $\theta^{*}(x,V)<\infty$, then there exists a sequence $r_n\to 0$ and $ W\in \cV(\B_1)$ such that
$$
V_{x,r_n}\rightharpoonup W\in \cV(\B_1).
$$
We denote with \(TV(x, V)\) the space of all varifolds \(W\) obtained as above. 

A varifold $V\in \cV(M)$ is said rectifiable if there exists a $2$-rectifiable set $K$ and a function $\theta\in L^1(M;\R^+;\H^2\res K)$, such that 
\begin{equation}\label{rappresenta}
V=\theta \H^2 \res K \otimes \delta_{T_xK}.
\end{equation}

Moreover we say that a rectifiable varifold $V$ is integral if in the representation \eqref{rappresenta}, the density function $\theta$ is also integer valued.

The $weak^*$ topology on $\cV(M)$ is not metrizable. Nevertheless, in any closed ball 
$$\left\{V\in\cV (M) \, : \, \|V\|(M) \leq r \right \}$$
 of radius $r$, the $weak^*$ topology is metrizable by Banach–Alaoglu theorem: in particular, an example of metric $d$ which induces the $weak^*$ topology on $\cV$ is
 $$d(V_1,V_2) = \sup\{ V_1 (f ) - V_2 (f ) \, : \,  f\in W^{1,\infty}(G_2(M)) \cap C_c(G_2(M))\}.$$
 The balls of radius $r$ and center $V$ in this metric will be denoted by $U_r(V)$.

 \subsection{Integrands}
The anisotropic integrands  that we consider are \(C^3\) positive functions
$$
F: G_2(M) \longrightarrow \R^+:=(0,+\infty),
$$
for which there exist a positive constant \(\lambda\) such that
\begin{equation}\label{cost per area}
0 < \frac{1}{\lambda} \leq F(x,T) \leq \lambda<\infty\qquad\textrm{for all \((x,T)\in G_2(M)\).}
\end{equation}
Given a $2$-rectifiable set $K\subset M$ and $V \in \cV(M)$, we define: 
\begin{equation*}
\FF(K) := \int_{K} F(x,T_xK)\, d\H^2\res K \quad \mbox{and} \quad \FF(V) := \int_{G_2(M)} F(x,T)\, dV(x,T).
\end{equation*}
We will often identify a $2$-rectifiable set with the canonically associated density one varifold supported on $K$.
For a vector field \(X\in \mathcal X(M)\), we consider a one-parameter family of diffeomorphisms of \(M\) into itself \(\{\varphi_t\}_{t\in \R}\), such that $\frac{d\varphi_t}{dt} =X$. The {\em anisotropic first variation} is  defined as the following linear operator:  
\[
\delta_{\FF} V(X):=\frac{d}{dt}\FF\big ( (\varphi_t)_{\#}V\big)\Big|_{t=0}.
\]

We say that a varifold \(V\in \mathbb \cV(M)\) has locally bounded anisotropic first variation if $\delta_{\FF} V$ is a Radon measure on $\Omega$, i.e. if
$$|\delta_{\FF} V(X)|\leq C(K)\|X\|_{\infty}, \quad \text{ for all $X \in  \mathcal X_c(M)$}.$$

\begin{Remark}\label{duality}
We identify  the integrand \(F:G_2(M) \longrightarrow \R^+\) with a positively one homogeneous even function $G:T M\to \R^+$ via the equality
\begin{equation}\label{e:identification}
G(x, r  \nu):=|r| F(x,  \nu^{\perp}) \qquad\textrm{for all \(r \in \R\) and \((x,\nu)\in SM\)},
\end{equation}
where $SM:=\{(x,\nu)\in TM: |\nu|=1\}$.
Note that \(G\in C^3(TM\setminus \{(x,0): x\in M\})\) and that  by one-homogeneity: 
\begin{equation}\label{eulercod1}
\langle D_{\nu} G(x, \nu),\nu\rangle = G(x,\nu)\qquad\mbox{for all \((x,\nu)\in TM\), with $\nu\neq 0$.}
\end{equation}
With these identifications, if $M=\R^3$ it is a simple calculation to check that:
\begin{equation*}
\begin{split}
\delta_{\FF} V(X)&=\int_{SM} \langle D_x G(x,\nu), X(x)\rangle\, dV(x,\nu)\\
& \quad + \int_{SM}\Big(G(x,\nu)\Id-\nu \otimes D_\nu G(x,\nu)\Big):DX(x)\, dV(x,\nu), \quad \forall X\in \mathcal X(M),
\end{split}
\end{equation*}
see for instance~\cite[Section 3]{Allard84BOOK} or ~\cite[Lemma A.4]{dephilippismaggi2}. 
\end{Remark}

We will always assume that $F$ is an elliptic integrand, i.e. the map $G$ associated to $F$ satisfies (up to select a bigger $\lambda$ in \eqref{cost per area}):
\begin{equation}\label{costper}
|D_x G|,|D_\nu G|,|D_{x\nu}^2G|,|D_\nu^2G|\leq \lambda, \quad |D_\nu G|\geq \frac1\lambda \quad \mbox{and} \quad D_\nu^2 (G_{|M\times \mathbb S^2})(x,\nu)\geq \frac{\Id_{\nu^\perp}}{\lambda}.
\end{equation}
The last condition above (the uniform convexity in all but the radial directions) in particular implies the strict convexity of $G$ in all (but the radial) directions:
\begin{equation}\label{eq:sconv}
G(x,\nu)> \langle D_\nu G(x,\bar \nu),\nu\rangle\qquad \mbox{for all \((x,\bar \nu),(x,\nu)\in SM\) with  \(\nu\ne \pm \bar \nu\).}
\end{equation}
\begin{Remark}\label{rectif}
We remark that condition \eqref{eq:sconv} is necessary and sufficient to apply the rectifiability theorem the authors proved in \cite[Theorem 1.2]{DDG2}, cf. also \cite{DRK}.
\end{Remark}

\subsection{CMC surfaces}
We will use the following definitions and tools, borrowed from \cite{ZZ}. We refer to \cite[Section 2]{ZZ} for more detailed notation.

We denote with $\C(M)$ the space of finite perimeter sets in $M$, also referred to as Caccioppoli sets. Given $c>0$, we define the following energy functional:
\begin{equation*}
\Fc(\Om)=\FF(\partial\Om)-c\haus^{3}(\Om), \qquad \forall \Omega \in \C(M).
\end{equation*}
In case $\partial\Om$ is a smooth surface, the first variation of $\Om$ with respect to $\Fc$ is 
\begin{equation}
\label{1stvFc}
\de_{\Fc}\Om(X)=\int_{\partial\Om}(H_F^{\partial \Om}-c)\langle X, \nu \rangle \, d\haus^2, \qquad \forall X\in \mathcal X(M),
\end{equation}
where $\nu$ and $H_F^{\partial \Om}$ denote respectvely the outward unit normal on $\partial \Om$ and the anisotropic mean curvature of $\partial \Om$ with respect to $\nu$.

We deduce from \eqref{1stvFc} that for any critical point $\Om$ of $\Fc$,  $\partial \Om$ has constant anisotropic mean curvature $c$ with respect to the outward unit normal $\nu$. This allows us to compute the second variation $\delta^2_{\Fc}\Om(X,X)$ of a critical point $\Om$ with respect to $\Fc$ applied to a vector field $X\in \mathcal X(M)$ such that $X(x)=\varphi(x) D_\nu G(x,\nu(x)) $ on $\partial\Om$ where $\varphi\in C^\infty(\partial\Om)$, combining \cite[Section 1.5, Page 295]{Allard1983} or \cite[Lamma A.5]{dephilippismaggi2} with \cite[Remark 2.4, Proposition 2.5]{BDE}. In particular, $\partial\Om$ is $c$-stable if 
\begin{equation}\label{stabineq}
0\leq \delta^2_{\Fc}\Om(X,X)=S^{\partial \Om}_{F}(\varphi,\varphi):=\delta^2_{\FF}(\partial \Om)(X,X)-\int_{\partial \Om} Ric^M(\nu,\nu)G(x,\nu)^2\varphi^2 \, d\haus^2.
\end{equation}
Here $\delta^2_{\FF}(\partial \Om)(X,X)$ is the Euclidean second variation of $\FF$ and it is calculated in \cite[Section 1.5, Page 295]{Allard1983} (for general $F$) or in \cite[Lemma A.5]{dephilippismaggi2} (without space dependence). The second term $-\int_M Ric^M(\nu,\nu)G(x,\nu)^2\varphi^2$ comes instead from the metric of $M$, as computed in \cite[Remark 2.4, Proposition 2.5]{BDE}, where $G(x,\nu)^2\varphi^2=\langle X, \nu\rangle^2$ as a simple application of the one-homogeneity of $G$. We explicitly observe that \eqref{stabineq} does not depend on $c$.
For a more detailed study about anisotropic CMC surfaces in the Euclidean space, we refer the author to \cite{DMMN, DDG, DeRosa, DG2, DG, DKS,DRL,DRT,HTi}.

\begin{Remark}\label{metricinG}
We remark that we can absorb the metric of $M$ in the elliptic integrand $G$. Hence, up to modify $G$, we can always assume to work with the flat metric. This remark will be particularly useful in all the local arguments of this paper, where we can consequently always assume to work in $\R^3$ rather than $M$. In particular, in this way the stability inequality \eqref{stabineq} simply reads
$$
\delta^2_{\FF}(\partial \Om)(X,X)\geq 0,
$$
where $\delta^2_{\FF}(\partial \Om)(X,X)$ is the Euclidean second variation of $\FF$ \cite[Section 1.5, Page 295]{Allard1983}.
\end{Remark}

\begin{Remark}\label{palla}
We remark that in the local arguments, by Remark \ref{metricinG}, when we work in $\R^3$ condition \eqref{costper} implies that $H_F^{\partial \B_r(x)}\geq \frac{1}{\lambda r}$ for every $x\in \R^3$ and every $r>0$.
\end{Remark}

\begin{Definition}
\label{D:stable c-hypersurface}
Consider an immersed, smooth, two-sided surface $\Sigma$ with unit normal vector $\nu$, and an open set $U\subset M$. $\Sigma$ is called a {\em $c$-stable surface} in $U$ if the anisotropic mean curvature $H_F^{\Sigma}$ of $\Sigma\cap U$ with respect to $\nu$ is identically equals to $c$ and $S^\Sigma_F(\varphi, \varphi)\geq 0$  for all $\varphi\in C^\infty(\Si)$ with $\spt (\varphi)\subset \Sigma\cap U$.

If $c=0$ we simply say that $\Sigma$ is a {\em stable surface}.
\end{Definition}

\begin{Remark}\label{stablee}
Since $|Ric^M(\nu,\nu)|G(x,\nu)^2\leq C(M,F)$, for every $c$-stable surface we deduce the validity of \cite[Equation (41)]{DDH}, applying \cite[Lemma 2.1]{Allard1983} or \cite[Lemma A.5]{dephilippismaggi2}.
\end{Remark}

\begin{Definition}
Consider an open subset $U\subset M$, and a smooth $2$-dimensional surface $\Sigma$. A smooth immersion $\psi: \Si\rightarrow U$ is an {\em almost embedding} if for every $p\in\phi(\Si)$ where $\Si$ fails to be embedded, there exists a neighborhood $Q\subset U$ of $p$, such that 
\begin{itemize}
\item $\Si\cap \psi^{-1}(Q)=\cup_{i=1}^n \Si_i$, where $\Si_i$ are disjoint connected components;
\item $\psi(\Si_i)$ is an embedding for every $i=1,\dots,n$; 
\item for every $i$, every $\psi(\Si_j)$, $j\neq i$, lies on one side of $\psi(\Si_i)$ in $Q$.
\end{itemize}
We identify $\psi(\Si)$ with $\Si$ and $\phi(\Si_i)$ with $\Si_i$. We define the touching set $\mS(\Si)$ as the set of points of $\Si$ where $\Si$ fails to be embedded. We define the regular set as $\mR(\Si):=\Si\backslash\mS(\Si)$.
\end{Definition}

\begin{Theorem}
\label{T:curvat}
Consider an open subset $U\subset M$. Let $\Si\subset U$ be a $c$-stable surface in $U$ such that $\partial \Si\cap U=\emptyset$, and $\H^2(\Si)\leq D$, then there exists $C>0$ depending only on $M, c, D, F$, such that
\begin{equation}\label{curvest}
|A^\Si|^2 (x) \leq \frac{C}{\d(x,\partial U)^2} \quad \text{ for all $x \in \Si$}, 
\end{equation}
where $A^\Si(x)$ denotes the second fundamental form of $\Sigma$ at $x$.

Moreover if $\Si_k\subset U$ is a sequence of $c$-stable surfaces in $U$ such that $\partial \Si_k\cap U=\emptyset$ and $\sup_{k} \H^2(\Sigma_k) < \infty$, 
then up to a subsequence, $\Sigma_k$ converges locally smoothly to a $c$-stable surface in $U$. 
\end{Theorem}
\begin{proof}
From Remark \ref{stablee} we deduce that $c$-stable surfaces satisfy all the assumptions in \cite{Allard1983} to get the curvature estimate \eqref{curvest}, provided they enjoy upper density estimates, cf. also \cite{White1987}. The upper density estimates for $c$-stable surfaces can be obtained again from Remark \ref{stablee}, i.e. the validity of \cite[Equation (41)]{DDH}, applying \cite[Lemma 3.5]{BT}, compare with the proof of \cite[Lemma 4.3]{DDH} for the estimate of the $L^2$-norm of the isotropic mean curvature $\int |H|^2$. We deduce that \eqref{curvest} holds.
The compactness statement can be easily deduced from the curvature estimates, see \cite[Theorem 2.6]{ZZ}. 
\end{proof}

\begin{Theorem}
\label{T:compact}
Given a sequence of almost embedded, $c_k$-stable surfaces $\Si_k\subset U$, such that $\sup_{k} \H^2(\Sigma_k) < \infty$ and $\sup_k c_k <\infty$. Then the following hold:
\begin{itemize}
\item[(i)] if $\inf c_k>0$, then $\{\Sigma_k\}$ converges locally smoothly to an almost embedded $c$-stable surface $\Sigma$ in $U$ (for some $c>0$), after possibly passing to a subsequence; moreover if $\{\Si_k\}$ are all boundaries, then the density of $\Sigma$ is $1$ on $\mR(\Si)$ and $2$ on $\mS(\Si)$, and $\Sigma$ is a boundary as well;
\item[(ii)] if $c_k\to 0$, then $\{\Sigma_k\}$ converges locally smoothly with integer multiplicity to a smooth embedded stable surface $\Sigma$ in $U$, after possibly passing to a subsequence.
\end{itemize}
\end{Theorem}
\begin{proof}
The proof is obtained repeating verbatim the proof of \cite[Theorem 2.11]{ZZ}, replacing the use of \cite[Theorem 2.6]{ZZ} with Theorem \ref{T:curvat}.
\end{proof}

\subsection{Regularity for minimizers of $\Fc$}
\begin{Theorem}
\label{T:regularity of Ac minimizers}
Let $\Om\in\C(M)$, $p\in\spt\|\partial\Om\|$, and $r>0$, such that $\Om$ minimizes the $\Fc$-functional in $B_r(p)$: that is, for every $W\in\C(M)$ with $(W\setminus \Om)\cup(\Om\setminus W)\subset B_r(p)$, we have $\Fc(W)\geq \Fc(\Om)$. Then $\partial\Om\res B_r(p)$ is a smooth, embedded surface.
\end{Theorem}
\begin{proof}
Since for every $U\in\C(M)$ with $(U\setminus \Om)\cup(\Om\setminus U)\subset \tilde B_r(p)$ we have $\Fc(U)\geq \Fc(\Om)$, then
\[ \FF(\partial U)-\FF(\partial\Om)\geq -c|\H^{3}(U)-\H^{3}(\Om)|\geq  -c\H^{3}((U\setminus \Om)\cup(\Om\setminus U)). \]
This is \cite[Condition (1.12) in Definition 1.8]{DM}. We deduce the claimed regularity applying \cite[Theorem 1.2, Remark 1.3]{DM}.
\end{proof}

\subsection{Sweepouts}
We recall the following notions of generalized smooth families and sweepouts, cf. \cite[Definition 0.2]{DLTas}:
\begin{Definition}\label{d:sweep}
Consider a family $\{\Sigma_t\}_{t\in [0,1]^k}$ of closed subsets of $M$.
We say that $\{\Sigma_t\}_{t\in [0,1]^k}$ is a {\em generalized smooth family} if the following properties hold
\begin{itemize}
\item $\cH^2(\Sigma_t)<\infty$ for every $t$;
\item For every $t$ there exists a finite $P_t\subset M$ such that $\Sigma_t$ is a smooth surface in $M\setminus P_t$;
\item $\FF (\Sigma_t)$ depends continuously on $t$ and  if \(t \to t_0\), $\sup_{x \in \Sigma_t} \d(x,\Sigma_{t_0}) \to 0$;
\item $\Sigma_t \longrightarrow \Sigma_{t_0}$ in $C^2$-norm as $t\rightarrow t_0$, in any $U\subset\subset M\setminus P_{t_0}$. 
\end{itemize}
A family $\{\Omega_t\}_{t\in [0,1]}$ 
of open finite perimeter sets is a {\em sweepout} of
$M$ if $\{\partial \Om_t\}_{t\in [0,1]}$ is a generalized smooth family and
\begin{itemize}
\item[(so1)] $\Omega_0=\emptyset$ and $\Omega_1 = M$; 
\item[(so2)] ${\haus^3} (\Omega_t\Delta \Omega_{t_0})
 \to 0$ 
as $t\to t_0$.
\end{itemize}
\end{Definition}

We recall the existence of sweepouts as in Definition \ref{d:sweep}, stated in \cite[Proposition 0.4]{DLTas}:
\begin{Proposition}[{see \cite[Proposition 0.4]{DLTas}}]\label{Mor}
Given any smooth Morse function $g: M \to [0,1]$, then $\{\{g\leq t\}\}_{t\in [0,1]}$ is a sweepout.
\end{Proposition}

\section{Min-max construction}
For every $c> 0$ and every sweepout $\{\Om_t\}$ we define
\begin{equation*}
 \cF(\{\Om_t\})\;:=\;\max_{t\in[0,1]} \Fc(\Om_t).
\end{equation*}
One can prove a uniform lower bound for $\cF$ on the sweepouts:

\begin{Proposition}\label{p:isop} There exists $C(M,F,c)>0$ such that $\cF (\{\Om_t\}) \geq C (M,F,c)$ for every sweepout $\{\Om_t\}$. 
\end{Proposition}
\begin{proof} Since $\{\Omega_t\}$ satisfies the properties of Definition \ref{d:sweep}, for every $V\in [0,{\haus^3}(M)]$ there is $t_0\in [0,1]$ such that ${\haus^3} (\Omega_{t_0})
= V$. By \cite[Theorem 2.15]{ZZ}, choosing $V\in [0,V_0]$, we compute 
$$\FF(\partial \Om_{t_0})\;\geq\; \frac{\cH^2 (\partial \Om_{t_0})}{\lambda}\;\geq\;  \frac{C_0 V^{\frac{2}{3}}}{\lambda},$$ 
where $C_0$ is the isoperimetric constant of $M$ in \cite[Theorem 2.15]{ZZ}. We can choose $V\in  [0,{\haus^3}(M)]$ such that $V= \min\{(\frac {C_0}{2c\lambda})^{3},V_0\}$, we deduce that
$$\Fc(\Om_{t_0})=\FF(\partial \Om_{t_0})-cV\geq \frac{C_0 V^{\frac{2}{3}}}{\lambda}-cV\geq cV= c \min\left \{\left(\frac {C_0}{2c\lambda}\right)^{3},V_0\right \}=:C(M,F,c),$$ 
where $C(M,F,c)$ depends just on $M$, $F$ and $c$.
\end{proof}

For every family $\mathscr L$ of sweepouts, we set
\begin{equation*}
 m_c(\mathscr L)\;:=\;\inf_\mathscr L \cF\;=\;
\inf_{\{\Om_t\}\in\mathscr L}\left[\max_{t\in[0,1]} 
\Fc(\Om_t)\right].
\end{equation*}
By Proposition \ref{p:isop}, $m_c (\mathscr L)\geq C(M,F,c)>0$.
We say that a sequence $\{\{\Om_t\}^k\}\subset\mathscr L$ is minimizing if 
$$
\lim_{k\to\infty}\cF(\{\Om_t\}^k)\;=\;m_c(\mathscr L)\, .
$$ 
A sequence  $\{\Om_{t_k}^k\}$ is a min-max sequence if $\{\{\Om_t\}^k\}$ is minimizing and  $\Fc(\Om_{t_k}^k)
\to m_c(\mathscr L)$. 
\begin{Remark}\label{uniform}
We observe that 
$$m_c(\mathscr L)\leq \inf_{\{\Om_t\}\in\mathscr L}\left[\max_{t\in[0,1]} 
\FF(\Om_t)\right]=:m_0(\mathscr L)<\infty, \qquad \forall c>0.$$
\end{Remark}
We will focus our study on the following families of sweepouts:

\begin{Definition}\label{d:homotopy}
Two sweepouts $\{\Om^0_s\}$, $\{\Om^1_s\}$ are homotopic
if there exists a generalized smooth family $\{\Om_t\}_{t\in [0,1]^2}$
such that $\Om_{(0,s)} = \Om^0_s$ and $\Om_{(1,s)}=\Om^1_s$.
A family $\mathscr L$ of sweepouts is {\em homotopically closed} if 
it contains the homotopy class of each of its elements.
\end{Definition}

The main result of this paper is the following:

\begin{Theorem}\label{t:main} Given $c>0$, for any homotopically closed family 
$\mathscr L$ of sweepouts there is a min--max sequence
$\{\Om_{t_k}^k\}$ converging (in the sense of varifolds)
to a non trivial surface $\Sigma$ with multiplicity one, which is smooth and almost embedded outside of one point $p\in M$ and $H^\Sigma_F\equiv c$. Moreover $\FF(\Sigma)\leq 2(m_c(\mathscr L)+c\haus^3(M))$ and one of the following properties hold:
\begin{itemize}
\item[(a)] there exists $R_{\Sigma}>0$ such that $\Sigma$ is smooth stable in $B_{R_\Sigma}(x)$ for every $x\in  M$ and there exists $y\in M$ such that $\Sigma$ is smooth stable in $M\setminus \overline{B_{18R_\Sigma}(y)}$;
\item[(b)] denoting $R_M:=\frac 12\min\left \{\lambda/c,  \Inj(M)/18, \frac{1}{\lambda(c+\lambda+4\lambda^3)}\right \}$, then $\Sigma$ is smooth stable in $B_{R_M}(x)$ for every $x\in  M$ (when $c\to 0$ then $R_M$ depends just on $F$ and $M$).
\item[(c)] $\Sigma$ is stable in $M\setminus \{p\}$.
\end{itemize}
\end{Theorem}

Since Morse functions exist on every smooth compact Riemannian manifold without boundary, \cite[Corollary 6.7]{Mi}, Proposition \ref{Mor} and Theorem \ref{t:main} provide a proof for Theorem \ref{t:existenceCMC}.
We also remark that Theorem \ref{t:existence} is a corollary of Theorem \ref{t:main}, as we show below.
\begin{proof}[Proof of Theorem \ref{t:existence}]
Fix an homotopically closed family $\mathscr L$ of sweepouts. Consider the sequence $c_k=\frac 1k$. Applying Theorem \ref{t:main}, we construct a sequence of non trivial surfaces $\Sigma_k$ with multiplicity one, which are smooth and almost embedded outside of a point $p_k\in M$, with constant anisotropic mean curvature $c_k$, such that $\FF(\Sigma_k)\leq 2(m_{c_k}(\mathscr L)+c_k\haus^3(M))$. In particular, by \eqref{cost per area} and Remark \ref{uniform}, $\sup_{k} \H^2(\Sigma_k)< \infty$.
Moreover one of the following properties hold:
\begin{itemize}
\item[(i)] there exists $R>0$ such that up to subsequences $R_{\Sigma_k}\geq R$ for every $k$;
\item[(ii)] $R_{\Sigma_k}\to 0$;
\item[(iii)] $\Sigma$ is smooth stable in $B_{R_M}(x)$ for every $x\in  M$;
\item[(iv)] $\Sigma_k$ is stable in $M\setminus \{p_k\}$.
\end{itemize}
In case (i), by Theorem \ref{T:compact}(ii), we deduce that $\{\Sigma_k\}$ converges locally smoothly (with multiplicity) to some smooth, embedded and stable surface in every $B_R(x)$. By the arbitrarity of $x\in M$, we conclude the proof.\\
In case (ii), for every $k$ there exists $y_k \in M$ such that $\Sigma_k$ is stable in $M\setminus \overline{B_{18R_{\Sigma_k}}(y_k)}$. By compactness of $M$, up to pass to a non-relabeled subsequence, $y_k\to p\in M$. In particular, since $R_{\Sigma_k}\to 0$, for every compact set $K\subset M\setminus \{p\}$, there exists $N\in \N$ such that for every $k\geq N$ then $\Sigma_k$ is smooth, $c_k$-stable and almost embedded in $K$.
By Theorem \ref{T:compact}(ii), we deduce that $\{\Sigma_k\}$ converges locally smoothly (with multiplicity) to some smooth, embedded and stable surface in $\mbox{int}(K)$. By the arbitrarity of $K$, we conclude the proof.\\
In case (iii), we can argue as in case (i).\\
In case (iv), by compactness of $M$, up to pass to a non-relabeled subsequence, $p_k\to p\in M$. In particular, for every compact set $K\subset M\setminus \{p\}$, there exists $N\in \N$ such that for every $k\geq N$ then $\Sigma_k$ is smooth, $c_k$-stable and almost embedded in $K$.
By Theorem \ref{T:compact}(ii), we deduce that $\{\Sigma_k\}$ converges locally smoothly (with multiplicity) to some smooth, embedded and stable surface in $\mbox{int}(K)$. By the arbitrarity of $K$, we conclude the proof.
\end{proof}

\section{Proof of Theorem \ref{t:main}} \label{s:overview}
This section is devoted to the proof of Theorem \ref{t:main}.
\subsection{Pull-tight}\label{sect1}
We aim to show the existence of a minimizing sequence $\{\{\Omega_t\}^k\}$ such that the boundaries of any corresponding min--max sequence converge
to a varifold with anisotropic first variation bounded by $c$. Nowadays, it is a well known construction for the isotropic case and it is referred to as pull-tight \cite[Section 4]{CD}, \cite[Section 4.3]{P}. We adapt the pull-tight in \cite{CD,P} to the anisotropic setting for the sake of exposition. In order to state it, we need
further terminology.

We denote
$$\cV:= \left\{V\in\cV (M) \, : \, \|V\|(M) \leq 2\lambda(m_c(\mathscr L)+c\haus^3(M))\right \}.$$
We set 
 $$\cV_\infty^c:=\{V \in \cV:|\delta_{\FF}V(X)| \le c\int  |X|\,d\|V\|\text{ for all } X \in \mathcal X_c(M))\},
$$
to be the set of varifolds with anisotropic mean curvature bounded by $c$.
$\cV_\infty^c$ is clearly closed by lower semicontinuity of the anisotropic first variation with respect to varifold convergence.

\begin{Proposition}\label{p:stationary}
Let $\mathscr L$ be a homotopically closed family of sweepouts. 
There exists a minimizing sequence $\{\{\Om_t\}^k\}
\subset\mathscr L$ such that, if $\{\Om_{t_k}^k\}$ 
is a min-max sequence, then $\cD(\partial \Om_{t_k}^k,\cV_\infty^c)\to 0$.
\end{Proposition}

\begin{proof}
\emph{Step 1: Mapping $\cV$ to the space of vector fields.}
For every $l \in \Z$ we define the annulus
$$\cV_l=\{V \in \cV \, : \, 2^{-l+1}\geq d(V,\cV_\infty^c)\geq 2^{-l-2}\}.$$
The sets $\cV_l$ are compact in the weak$^*$ topology by Banach-Alaoglu Theorem.

For every $V \in \cV_l$, there exists by definition a smooth vector field
$$X_V \in \mathcal X (M), \quad \mbox{ s.t. }  \, \delta_\FF V(X_V)+c\int  |X|\,d\|V\|<0.$$
If $l\in \N\setminus\{0\}$, up to multiply $X_V$ by a suitable constant, we can also assume
$$\|X_V\|_{C^l}\leq \frac 1l.$$
By continuity of the functional $Z\mapsto \delta_\FF Z(X_V)+c\int  |X_V|\,d\|Z\|$ with respect to the $weak^*$ topology for varifolds, we get that for every $V \in \cV_l$, there exists $\eps_V>0$ such that
\begin{equation}\label{cover}
\delta_\FF W(X_V)+c\int  |X_V|\,d\|W\|\leq \frac 12 \left( \delta_\FF V(X_V)+c\int  |X_V|\,d\|V\|\right), \quad \forall W\in U_{2\eps_V}(V).
\end{equation}
By compactness of $\cV_l$, we can cover it with a finite number $N(l)$ of balls $U_{\eps_{V_i^l}}(V_i^l)$ satisfying property \eqref{cover} (with vector fields $X_{V_i^l}$). 
Moreover, for each $i=1,\dots,N(l)$, we choose $\varphi_i^l\in C_c(U_{2\eps_{V_i^l}}(V_i^l))$ which is equal to $1$ on $U_{\eps_{V_i^l}}(V_i^l)$ and satisfies $0\leq \varphi_i^l\leq 1$.

We now define the following continuous function
$$H^l:V \in \cV_l \to  H_V^l \in \mathcal X (M) \qquad \mbox{where} \quad H_V^l:= \frac{\sum_{i=1}^{N(l)} \varphi_i^l(V) X_{V_i^l}}{\sum_{i=1}^{N(l)} \varphi_i^l(V)},$$
which by construction satisfies:
$$\delta_\FF V(H_V^l)+c\int  |H_V^l|\,d\|V\|<0 \mbox{ for every } V\in \cV_l.$$
We also remark that if $l\in \N\setminus\{0\}$, then $\|H^l_V\|_{C^l}\leq \frac 1l$ for every $V\in \cV_l$.

Next, for every $l \in \Z$, we choose a function $\psi^l$ with the properties that $0\leq \psi^l \leq  1$ and $ \psi^l =1$ on $\{V \in \cV \, : \, 2^{-l}\geq d(V,\cV_\infty^c)\geq 2^{-l-1}\}$. On $\cV\setminus \cV_\infty^c$ we define the continuous function 
$$H:V \in \cV\setminus \cV_\infty^c \to  H_V \in \mathcal X (M) \qquad \mbox{where} \quad H_V:= \frac{\sum_{l\in \Z} \psi^l(V) H^l_{V}}{\sum_{l\in \Z} \psi^l(V)},$$
which by construction satisfies
$$\delta_\FF V(H_V)+c\int  |H_V|\,d\|V\|<0 \mbox{ for every } V\in \cV\setminus \cV_\infty^c.$$
Moreover, if $d(V,\cV_\infty^c)\leq 2^{-l}$ with $l\in \N\setminus\{0,1\}$, then $\|H_V\|_{C^{l-1}}\leq \frac 1{l-1}$.
Hence, we can extend the map $V\to H_V$ to $\cV$ continuously in every $C^l$ norm by setting it equal to $0$ on $\cV_\infty^c$. 

\emph{Step 2: Mapping $\cV$ to the space of isotopies.}
For each $V\in \cV$ let $\Phi_V$ be the $1$-parameter family of diffeomorphisms generated by $H_V$, i.e.
$$\Phi_V:[0,+\infty)\times M\to M \quad \mbox{where} \quad \frac{\partial \Phi_V}{\partial t}(t,x)=H_V(\Phi_V(t,x)).$$
By continuity of the functional $Z\mapsto \delta_\FF Z(H_V)+c\int  |H_V|\,d\|Z\|$, for each $V\in  \cV\setminus \cV_\infty^c$ there is a positive time $\sigma_V$ such that, for every $s\in [0,\sigma_V]$
$$ \delta_\FF ((\Phi_V(s,\cdot))_\# V)(H_V)+c\int  |H_V|\,d\|(\Phi_V(s,\cdot))_\# V\|\leq \frac 12 \left(\delta_\FF V(H_V)+c\int  |H_V|\,d\|V\|\right)<0.$$
By the continuity of the map $H_V$, the map 
$$(s,V)\mapsto  \delta_\FF ((\Phi_V(s,\cdot))_\# V)(H_V)+c\int  |H_V|\,d\|(\Phi_V(s,\cdot))_\# V\|$$
 is also continuous. Thus we conclude the existence of a radius $\rho_V$ such that, for every $s\in [0,\sigma_V]$
 and $W\in U_{2\rho_V}(V)$
$$ \delta_\FF ((\Phi_W(s,\cdot))_\# W)(H_W)+c\int  |H_W|\,d\|(\Phi_W(s,\cdot))_\# W\|\leq \frac 14 \left(\delta_\FF V(H_V)+c\int  |H_V|\,d\|V\|\right)<0.$$
Similarly to Step 1, we can construct a continuous function $\sigma: \cV\to [0, \infty]$ such that $\sigma(V)=0$ for every $V\in \cV_\infty^c$, $\sigma(V)>0$ for every  $V\in \cV\setminus \cV_\infty^c$, and
\begin{equation}\label{nuov}
\max_{s\in [0,\sigma(V)]} \delta_\FF ((\Phi_V(s,\cdot))_\# V)(H_V)+c\int  |H_V|\,d\|(\Phi_V(s,\cdot))_\# V\|<0 \mbox{ for every } V\in \cV\setminus \cV_\infty^c.
\end{equation}
We can redefine a (non relabeled) $H_V$ by multiplying the old one by $\sigma(V)$. The new function $H_V$ remains continuous and vanishes identically on $\cV_\infty^c$, but \eqref{nuov} now reads
\begin{equation}\label{maxx}
\max_{s\in [0,1]} \delta_\FF ((\Phi_V(s,\cdot))_\# V)(H_V)+c\int  |H_V|\,d\|(\Phi_V(s,\cdot))_\# V\|<0, \qquad \forall V\in \cV\setminus \cV_\infty^c.
\end{equation}

\emph{Step 3: Construction of an intermediate minimizing sequence.}
We choose a minimizing sequence of families $\{\{U_t\}^k\}\subset \mathscr L$ and consider the new families $\{\{\tilde \Om_t\}^k\}$ defined as
$$\tilde \Om^k_t=\Phi_{\partial U_t^k}(1,U_t^k), \qquad \forall t\in [0,1], k \in \N.$$
Notice that this is not our final minimizing sequence, because $\{\{\tilde \Om_t\}^k\}$ is not necessarily an element of $\mathscr L$, since the map $(t, x)\mapsto \Phi_{\partial U_t^k}(1, x)$ is not known to be smooth in the parameter $t$.
Nevertheless, we prove that the sequence $\{\{\tilde \Om_t\}^k\}$ satisfies the property claimed by the proposition. 
By \eqref{maxx}, we know that for every $t\in [0,1]$ and $k \in \N$
\begin{equation}\label{maxxx}
\begin{split}
\FF^c(\tilde \Om^k_t) - \FF^c(U_t^k) &\leq \int_0^1\delta_{\FF^c} (\Phi_{\partial U_t^k}(s,U_t^k))(H_{\partial U_t^k})\, ds\\
&\leq \int_0^1\left(\delta_\FF (\Phi_{\partial U_t^k}(s,\partial U_t^k))(H_{\partial U_t^k})+c\int  |H_{\partial U_t^k}|\,d\|\Phi_{\partial U_t^k}(s,\partial U_t^k)\|\right)\, ds\leq 0.
\end{split}
\end{equation}
Since $\{\{U_t\}^k\}\subset \mathscr L$ is a minimizing sequence, \eqref{maxxx} implies that also $\{\{\tilde \Om_t\}^k\}$ is a minimizing sequence. Notice that we are making a slight abuse of notation from now until the end of the step, since $\{\{\tilde \Om_t\}^k\}$ is not necessarily a subset of $\mathscr L$. We still call it minimizing sequence to denote that $\limsup_{k\to\infty}\cF(\{\tilde \Om_t\}^k)\;\leq\;m_c(\mathscr L).$ We will use the same abuse of notation for a min-max sequence $\{\tilde \Om_{t_k}^k\}$, simply meaning that $\lim_{k\to\infty} \FF^c(\tilde \Om_{t_k}^k)=m_c(\mathscr L)$.

We now want to show that every min-max sequence associated to $\{\{\tilde \Om_t\}^k\}$ clusters to $\cV_\infty^c$. The main idea of the proof is to show that the min-max sequences associated to $\{\{\tilde \Om_t\}^k\}$ are generated through $\Phi$ just by those min-max sequences of $\{\{U_t\}^k\}$ clustering to $\cV_\infty^c$.
Indeed, fix a general sequence $\{t_k\}$ such that $\lim_{k\to\infty} \FF^c(\tilde \Om_{t_k}^k)=m_c(\mathscr L)$. By \eqref{maxxx}, we deduce that $\partial U_{t_k}^k$ is a min-max sequence associated to $\{\{U_t\}^k\}$. In particular, up to extract a non relabeled subsequence, 
$$\haus^2(\partial  U_{t_k}^k)\leq \lambda\FF(\partial U_{t_k}^k) \leq 2\lambda(m_c(\mathscr L)+c\haus^3(M)).$$
This uniform bound implies that, up to possibly passing to a further subsequences, $\partial U_{t_k}^k$ converges to some varifold V. In particular, by continuity of the maps $\Phi$, we deduce that
$$\tilde \Om^k_{t_k}=\Phi_{\partial U_{t_k}^k}(1,U_{t_k}^k) \rightharpoonup \Phi_{V}(1,\cdot)_\# V, \qquad \mbox{in the sense of varifolds}.$$
We claim that $V \in  \cV_\infty^c$, otherwise if $V \in  \cV \setminus \cV_\infty^c$, then we would compute the following contradiction:
\begin{equation*}
\begin{split}
&m_c(\mathscr L)=\lim_{k\to \infty}\FF^c(\tilde \Om_{t_k}^k)=\lim_{k\to \infty} \FF^c(U_{t_k}^k)+\lim_{k\to \infty}\int_0^1\delta_{\FF^c} (\Phi_{\partial U_{t_k}^k}(s,U_{t_k}^k))(H_{\partial U_{t_k}^k})\, ds\\
&\leq m_c(\mathscr L)+ \lim_{k\to \infty}  \int_0^1 \left(\delta_\FF ((\Phi_{\partial U_{t_k}^k}(s,\cdot))_\# \partial U_{t_k}^k)(H_{\partial U_{t_k}^k})+c\int  |H_{\partial U_{t_k}^k}|\,d\|(\Phi_{\partial U_{t_k}^k}(s,\cdot))_\# \partial U_{t_k}^k\|\right)ds\\
&=m_c(\mathscr L)+  \int_0^1 \left(\delta_\FF ((\Phi_{V}(s,\cdot))_\# V)(H_{V})+c\int  |H_{V}|\,d\|(\Phi_{V}(s,\cdot))_\# V\|\right)ds \overset{\eqref{maxx}}{<}m_c(\mathscr L).
\end{split}
\end{equation*}
Since $V \in  \cV_\infty^c$, then $H_V=0$ and consequently the diffeomorphism $ \Phi_{V}$ is just the identity. We deduce that  $\tilde \Om^k_{t_k} \rightharpoonup \Phi_{V}(1,\cdot)_\# V=V \in \cV_\infty^c$. By the arbitrarity of the min-max sequence $\{\tilde \Om^k_{t_k}\}$ we conclude the proof of this step.

\emph{Step 4: Construction of the final minimizing sequence.}
We now wish to construct  the final minimizing sequence $\{\{\Om_t\}^k\}$, which still satisfies the property of the proposition and is contained in $\mathscr L$. To this aim we want to regularize each  $\{\{U_t\}^k\}$ in the parameter $t$. 

For each $k\in \N$, let $h^k_t$ denote the one parameter family of vector fields $H_{\partial U_{t}^k}$. The map $(t, x) \mapsto h^k_t(x)$ is continuous. Moreover, for every fixed $l\in \N$
$$
\lim_{t\to \tau}\|h^k_t(\cdot)-h^k_{\tau}(\cdot)\|_{C^l}=\lim_{t\to \tau}\|H_{\partial U_{t}^k}-H_{\partial U_{\tau}^k}\|_{C^l}=0.
$$
Convolving $h^k_t$ with a standard convolution kernel in the parameter $t$, we can construct a smooth map $(t, x) \mapsto \bar h^k_t(x)$ with the property that 
\begin{equation}\label{maxxxx}
\max_{t}\|h^k_t(\cdot)-\bar h^k_{t}(\cdot)\|_{C^1}\leq \frac 1{k+1}.
\end{equation}
Consider now for each fixed $k$ and $t$ the one-parameter family of diffeomorphisms $\Psi^k_{t}(s,\cdot)$ generated by $\bar h^k_{t}$  and the one-parameter family of diffeomorphisms $\Phi^k_{t}(s,\cdot)$ generated by $h^k_{t}$. Recall that $\tilde \Om_{t}^k=\Phi^k_{t}(1,U_{t}^k)$. We define the corresponding family $\Om_{t}^k=\Psi^k_{t}(1,U_{t}^k)$. By the smoothness of the map $(t, x) \mapsto \bar h^k_t(x)$, we know that $\{\{\Om_t\}^k\}\subset \mathscr L$.

We first observe that \eqref{maxxxx} and \eqref{maxxx} imply that
$$\limsup_{k\to\infty}\max_t\Fc(\Om_t^k)=\limsup_{k\to\infty}\max_t\Fc(\tilde \Om_t^k)\leq \limsup_{k\to\infty}\max_t\Fc(U_t^k)\leq m_c(\mathscr L).$$
Consequently $\{\{\Om_t\}^k\}\subset \mathscr L$ is a minimizing sequence.

Furthermore, given a min-max sequence $\{\Om_{t_k}^k\}$ associated  to $\{\{\Om_t\}^k\}$, that is $\lim_{k\to\infty} \FF^c(\Om_{t_k}^k)=m_c(\mathscr L)$, again by \eqref{maxxxx} we have 
$\lim_{k\to\infty} \FF^c(\tilde \Om_{t_k}^k)=m_c(\mathscr L)$. Therefore, by Step 3, $\lim_{k\to\infty}
d( \partial  \tilde\Om_{t_k}^k,\cV_\infty^c)$.
Since, again by \eqref{maxxxx}, 
$$\lim_{k\to \infty}\max_{t}\d(\partial \Om_t^k,\partial \tilde\Om_t^k)=0,$$
we deduce that $\lim_{k\to\infty} d(\partial  \Om_{t_k}^k,\cV_\infty^c)$, as desired.
\end{proof}

\subsection{Almost minimality}
The limiting varifolds of the min-max sequences obtained in Subsection \ref{sect1} are not necessarily regular. Hence, in order to prove Theorem \ref{t:main} we introduce the notion of
almost minimizing varifolds, cf. \cite[Definition 2.3]{DLTas}. 

\begin{Definition}\label{d:am}
 Fix $\eps>0$ and an open set $U\subset M$. An open set
$\Omega \subset M$ is called 
\textit{$\eps$-almost minimizing} ($\eps$-a.m.) in $U$ if it does not exist any one parameter family of open sets $\{\Omega_t\}_{t\in [0,1]}$ such that:
\begin{eqnarray}
&&\{\partial \Omega_t\}_{t\in [0,1]} \mbox{is a generalized smooth family and (so2)
of Definition \ref{d:sweep} hold;}\label{m1}\\
&&\mbox{$\Omega_0=\Omega$ and $\Omega_t\setminus U =
\Omega\setminus U$ for all $t\in  [0,1]$;}\label{m2}\\
&&\mbox{$\FF^c(\Omega_t)\leq \FF^c(\Omega)+
\frac{\eps}{8}$ for all $t\in  [0,1]$;}\label{m3}\\
&&\mbox{$\FF^c(\Omega_1)\leq \FF^c(\Omega)-
\eps$.}\label{m4}
\end{eqnarray}
A sequence $\{\Omega^k\}$ of open sets is called 
{\em almost minimizing (a.m.) in $U$} if each 
$\Omega^k$ is $\eps_k$-a.m. 
in $U$, where $\eps_k$ possibly depends on $U$ and $\eps_k\to 0$ as $k\to \infty$. 
\end{Definition} 

The analogous notion for the isotropic stationary setting was introduced
by Pitts and then rephrased by Colding-De Lellis (see Section 3.2 of \cite{CD}) and by De Lellis-Tasnady (see Section 2.2 of \cite{CD}).
Using a combinatorial argument inspired by a general one of \cite{Alm} reported in \cite{P}, we prove the following existence result, following the isotropic counterpart in \cite[Proposition 2.4]{DLTas}.

\begin{Proposition}\label{p:almost1}
Let $\mathscr L$ be a homotopically closed family of sweepouts. 
There is a min-max sequence 
$\Om^k=\Om^k_{t_k}$ and a function $r: M \to (0,\infty]$ such that for every $An\in\An_{r(x)}(x)$ with $x \in M$, there exists a (non relabeled) subsequence $\{\Om^k\}$ that is a.m. in $An$ and such that $\partial \Om^k$ converges to a varifold $V\in \cV_\infty^c$, as 
$k\to \infty$.
Moreover the function $r(x)$ satisfies one of the following properties:
\begin{itemize}
\item[(a)] there exists $\Inj(M)/18>R>0$ such that $r(x)\equiv R$ for every $x\in M$ and there exists $y\in M$ such that $\{\Om^k\}$ is a.m. in $M\setminus \overline{B_{18R}(y)}$.
\item[(b)] $r(x)\equiv \Inj(M)/18$ for every $x\in M$.
\item[(c)] there exists $p\in M$ such that $r(x)=d(x,p)$ for $x\neq p$ and $r(p)=\infty$.
\end{itemize}
\end{Proposition}

The proof of Proposition
\ref{p:almost1} will build on the following result.

\begin{Lemma}\label{l}
Given two open sets $U\subset\subset \tilde U\subset M$ and a sweepout
$\{Q_t\}_{t\in [0,1]}$. Consider an $\eps>0$, $t_0\in [0,1]$ and a one parameter family
of open sets $\{\Omega_s\}_{s\in [0,1]}$ satisfying \eqref{m1}, \eqref{m2}, \eqref{m3}
and \eqref{m4}, with $\Omega_0 = Q_{t_0}$. Then there exists $\eta>0$ such that  for every $a,\alpha,\beta,b$ with
$t_0-\eta \leq a < \alpha<\beta<b\leq t_0+\eta$, the following holds:
There exists a sweepout $\{ Q'_t\}_{t\in [0,1]}$ homotopic to $\{\partial Q_t\}$ and satisfying:
\begin{itemize}
\item[(a)] $Q_t = Q'_t$ for every $t\in [0,a]\cup [b,1]$ and
$Q_t\setminus \tilde U=Q'_t\setminus \tilde U$ for every $t\in (a,b)$;
\item[(b)] $\FF^c (Q'_t)\leq 
\FF^c (Q_t) + \frac{\eps}{4}$ for every $t \in [0,1]$;
\item[(c)] $\FF^c (Q'_t)\leq \FF^c (Q_t)
-\frac{\eps}{2}$ for every $t\in (\alpha,\beta)$.
\end{itemize}
\end{Lemma}

Proposition \ref{p:almost1} can be proved gathering Lemma \ref{l} with a combinatorial
argument due to Almgren-Pitts. This statement is slightly stronger than the corresponding propositions for the isotropic case in \cite[Section 5]{CD} and \cite[Section 3]{DLTas}. We will need this stronger version to prove the regularity of the surface we construct away from a single point. The difficulty appears in Section \ref{regsec}: in the anisotropic setting we are not able to remove the center singularities in the punctured balls, and in principle we may have a finite number of singularities. But with the strengthened Proposition \ref{p:almost1}, we will be able to cover $M$ with balls such that the surface is regular in all the punctured balls and moreover each center, but $p$, is contained in another ball.

In order to prove  Proposition \ref{p:almost1}, we need some further notation. Below we recall \cite[Definition 3.2]{DLTas}:

\begin{Definition}
 Let $A,B\subset M$ be two open sets. We say that $
\Omega\in \C(M)$ is $\eps$-a.m. in $(A,B)$ if it is $\eps$-a.m. in at least one 
of the two open sets. A sequence $\{\Omega^k\}$ of open sets is called 
{\em almost minimizing (a.m.) in $(A,B)$} if each 
$\Omega^k$ is $\eps_k$-a.m. 
in $(A,B)$ for some sequence $\eps_k\to 0$ (possibly depending on $(A,B)$). 
We denote by $\cC\cO$ the set of pairs 
$(A,B)$ of open sets with
$$\d(A,B)\geq 4\min\{\diam(A),\diam(B)\}.$$
\end{Definition}

The following proposition is Almgren--Pitts combinatorial
Lemma: Proposition \ref{p:almost1} follows as a corollary
of it. 

\begin{Proposition}\label{p:combinatorial}
For every homotopically closed
 family of sweepouts $\mathscr L$, there exists a min-max 
sequence $\{\Om^n\}=\{\Omega^{k(n)}_{t_{k(n)}}\}$ 
such that $\partial \Om^n \rightharpoonup V \in \cV_\infty^c$ in the sense of varifolds, and $\{\Om^n\}$ is a.m. in $(A,B)$  for every $(A, B)\in\cC\cO$. 
\end{Proposition}

\begin{proof}[Proof of Proposition \ref{p:almost1}]
 We claim that the sequence $\{\Om^n\}$ given by Proposition \ref{p:combinatorial} satisfies the claim of 
Proposition \ref{p:almost1}. 
To this aim we fix $R>0$ such that $\Inj(M)>9R>0$. Then, $(B_R(x),M\setminus \overline{B_{9R} 
(x)})\in\cC\cO$ for all $x\in M$. 
In particular we deduce that $\{\Om^n\}$ is
a.m. in $(B_R(x),M\setminus \overline{B_{9R} (x)})$. Hence for every $R\in (0,\Inj(M)/9)$ one of the following cases hold:
\begin{itemize}
 \item [($a_R$)] $\{\Om^n\}$ is a.m. in $B_R(x)$ for every $x\in M$;
 \item [($b_R$)] there exists $p_R\in M$ and a (not relabeled) subsequence $\{\Om^n\}$ such that $\{\Om^n\}$ is a.m. in $M\setminus \overline{B_{9R} (p_R)}$.
\end{itemize}
Denote with $R':=\sup\{\Inj(M)/18>R\geq0 :\mbox{($a_R$) holds}\}$. If $\Inj(M)/18>R'>0$, then denoting $R'':=\frac 23 R$ we clearly have a sequence
as in Proposition \ref{p:almost1} (a) for which $r(x)\equiv R''$ for every $x\in M$ and there exists $y\in M$ such that $\{\Om^n\}$ is a.m. in $M\setminus \overline{B_{18R''}(y)}$. If $\Inj(M)/18=R'$, analogously  (b) holds.

The last possibility is $R'=0$, that is ($b_R$) holds for every $R\in (0,\Inj(M)/18)$. Then there exist a subsequence of $\{\Om^n\}$,
not relabeled, and a sequence of points $\{p_j\}_{j\in \N}\subset M$
such that $p_j\to p \in M$ as $j\to \infty$ and
\begin{center}
for any fixed $j$, there exists a subsequence $\{\Om^{n_j}\}$ that is a.m. in $M\setminus \overline{B_{1/j} (p_j)}$.
\end{center}
If $x\in M\setminus \{p\}$, for every $r<\d(x,p)$ then $B_{r}(x)\subset\subset M\setminus \{p\}$. Therefore, for every $An\in\An_{\d(x,p)}(x)$ there exists a (non relabeled) subsequence $\{\Om^n\}$ that is a.m. in $An$. Moreover, for every  $An\in\An(p)$, we have $An\subset\subset M\setminus \{p\}$, hence there exists a (non relabeled) subsequence $\{\Om^n\}$ that is a.m. in $An$. We deduce that  $\{\Om^n\}$ satisfies (c) of Proposition \ref{p:almost1}, which completes the proof.
\end{proof}

\begin{proof}[Proof of Proposition \ref{p:combinatorial}]
The proof of Proposition \ref{p:combinatorial} is obtained repeating verbatim the proof of \cite[Proposition 3.4]{DLTas}, with the only difference of replacing all the occurences of $\mathcal H^n$ and $\mathcal F$ in \cite[Proposition 3.4]{DLTas} respectively with $\FF^c$ and $\cF$, and replacing the use of \cite[Lemma 3.1]{DLTas} with Lemma \ref{l}.
\end{proof}

\begin{proof}[Proof of Lemma \ref{l}]
Although the proof of Lemma \ref{l} follows the same proof of \cite[Lemma 3.1]{DLTas}, the estimates need to be adapted to the functional $\FF^c$. Hence we sketch the argument for the sake of readability, we refer to the proof of \cite[Lemma 3.1]{DLTas} for more details.

We fix two open sets $A$, $B$ such that $U\subset\subset A \subset\subset B \subset\subset \tilde U$ and $\partial Q_{t_0}\cap C$ is a smooth surface, where $C=B\setminus \overline{A}$.
Moreover, we choose two functions $\varphi_A\in C^\infty_c (B)$, $\varphi_B\in C^\infty_c (M\setminus \overline{A})$ such that $\varphi_A+\varphi_B=1$. 
Next, we consider normal coordinates $(z,\sigma)\in \partial Q_{t_0}\cap C\times (-\delta, \delta)$
in a regular $\delta$--neighborhood of $C\cap \partial Q_{t_0}$. 
As $Q_t$ converges to $Q_{t_0}$, there exist $\eta>0$ and an open $C'\subset C$,
such that the following holds for every $t\in (t_0-\eta, t_0+\eta)$:
\begin{itemize}
\item $\partial Q_t\cap C$ is the graph of a function $\gamma_t$
over $\partial Q_{t_0}\cap C$;
\item $Q_t\cap C\setminus C' = Q_{t_0}\cap C\setminus C'$;
\item $Q_t\cap C' = \{(z,\sigma): \,\sigma<\gamma_t(z)\}\cap C'$.
\end{itemize}
 We also define
\begin{equation*}
\gamma_{t,s, \tau}\;:=\; \varphi_B \gamma_t + \varphi_A ((1-s) \gamma_t + s \gamma_{\tau})\, 
\qquad t, \tau\in (t_0-\eta, t_0+\eta), s\in [0,1].
\end{equation*}
Denoting with $\Gamma_{t,s, \tau}$ the graph of $\gamma_{t,s, \tau}$, there exists $\eta$ small enough such that
\begin{equation}\label{aaa1}
\max_{s,\tau} \FF (\Gamma_{t,s, \tau}) \;\leq\; \FF (\partial Q_t \cap C) + \frac{\eps}{32},
\end{equation}
and
\begin{equation}\label{aaa2}
 \max_{s,\tau} c\haus^3 (\{(z,\sigma): \sigma < \gamma_{t, s, \tau} (z)\}\cap C)\leq c\haus^3 (Q_t\cap C)+\frac{\eps}{32}\, .
\end{equation}
Now, given $t_0-\eta<a<\alpha<\beta<b<t_0+\eta$, we choose $\alpha'\in (a, \alpha)$ and
$\beta'\in (\beta,b)$ and consider a smooth function $\psi: [a,b]\to [0,1]$
that is equal to $0$ in a neighborhood of $a$ and $b$
and has value $1$ on $[\alpha', \beta']$. Moreover we choose a smooth function $\gamma: [a,b]\to [t_0-\eta, t_0+\eta]$ which
is the identity in a neighborhood of $a$ and $b$ and equal to
$t_0$ in $[\alpha', \beta']$.
Finally, we define the family of open sets $\{O_t\}$ satisfying $O_t = Q_t$ for $t\not\in [a,b]$,
while for every  $t\in [a,b]$ we impose:
$$O_t\setminus \overline{B} = Q_t\setminus \overline{B}, \qquad O_t\cap A = Q_{\gamma (t)}\cap A$$
$$O_t\cap C\setminus C'= Q_{t_0}\cap C \setminus C',\qquad 
O_t\cap C' = \{(z,\sigma): \sigma < \gamma_{t, \psi (t), \gamma(t)} (z)\}\cap C'.$$
In particular $\{\partial O_t\}$ is a sweepout homotopic to
$\partial Q_t$. By \eqref{aaa1} and \eqref{aaa2}, we estimate
\begin{equation}\label{aaa3}
\begin{split}
\FF^c (O_t\cap C) &\leq \max_{s,\tau} \FF (\Gamma_{t,s, \tau}) + \max_{s,\tau} c\haus^3 (\{(z,\sigma): \sigma < \gamma_{t, s, \tau} (z)\}\cap C)\\
&\;\leq\;  \FF^c ( Q_t\cap C)+\frac{\eps}{16}
\qquad \mbox{for $t\in [a,b]$.}
\end{split}
\end{equation}
We choose a smooth function $\chi: [\alpha', \beta']\to [0,1]$ which is equal to $0$ in
a neighborhood of $\alpha'$ and $\beta'$ and which is identically $1$ on $[\alpha,\beta]$. For $t\not\in [\alpha', \beta']$ we set $Q'_t = O_t$, while for $t\in [\alpha', \beta']$ we set:
$$Q'_t\setminus A = O_t\setminus A, \qquad  Q'_t\cap A = \Omega_{\chi (t)}\cap A.$$
$\{Q'_t\}$ is a sweepout homotopic
to $\{ O_t\}$ and hence to $\{Q_t\}$. To verify properties $(a)$, $(b)$ and $(c)$ of the lemma, we need to
estimate $\FF^c (Q'_t)$.

If $t\not\in [a,b]$, then $Q'_t \equiv Q_t$ and hence
$\FF^c (Q'_t) \;=\; \FF^c (Q_t)$, which in turn implies the validity of properties $(a)$, $(b)$.

We focus now on the more involved  $t\in [a,b]$. In this case, we have $Q'_t\setminus B=Q_t\setminus B$ and
$Q'_t\cap C= O_t\cap C$. This shows the property $(a)$ of the lemma. Moreover, we have
\begin{eqnarray}
\FF^c (Q'_t)
- \FF^c (Q_t) &\leq& [\FF^c (O_t\cap C)
- \FF^c (Q_t\cap C)] + [\FF^c (Q'_t \cap A)
 -\FF^c (Q_t\cap A)]\, \nonumber\\
&\stackrel{\eqref{aaa3}}{\leq}& \frac{\eps}{16} + 
[\FF^c (Q'_t \cap A) -\FF^c (Q_t\cap A)]\label{e:break}.
\end{eqnarray}
Now, we have to consider three subcases: 
\begin{itemize}
 \item[(Subcase 1)] If $t\in [a,\alpha']\cup [\beta',b]$, then $Q'_t\cap A = O_t
\cap A = Q_{\gamma (t)}\cap A$. Since $\gamma (t), t\in (t_0-\eta,
t_0+\eta)$, choosing $\eta$ small enough, we can assume
\begin{equation}\label{e:oscil}
|\FF^c(Q_s\cap A) - \FF^c (Q_\sigma \cap A)|
\;\leq\; \frac{\eps}{16} \qquad \mbox{for every $\sigma,s\in (t_0-\eta, t_0+\eta)$}.
\end{equation}
Hence \eqref{e:break} implies that $
\FF^c (Q'_t)
\;\leq\; \FF^c ( Q_t) + \frac{\eps}{8}$.
\item[(Subcase 2)] If $t\in [\alpha',\alpha]\cup [\beta',\beta]$, then $Q'_t\cap A
= \Omega_{\chi (t)}\cap A$. Hence, by \eqref{e:break}, we compute
\begin{eqnarray*}
\FF^c (Q'_t)-\FF^c (Q_t)
&\leq& \frac{\eps}{16}
+ [\FF^c (Q_{t_0}\cap A) -\FF^c (Q_t\cap A)]\nonumber\\
&&\quad\;\,+\, [\FF^c (\Omega_{\chi (t)}\cap A) - \FF^c (Q_{t_0}\cap A)]\stackrel{\eqref{e:oscil},\eqref{m3}}{\leq}
 \frac{\eps}{4}.
\end{eqnarray*}
\item[(Subcase 3)] If $t\in [\alpha,\beta]$, then $Q'_t\cap A = \Omega_1\cap A$. Hence, by \eqref{e:break}, we compute
\begin{eqnarray*}
\FF^c (Q'_t)-\FF^c (Q_t)
&\leq& \frac{\eps}{16}
+ [\FF^c (\Omega_1\cap A) -\FF^c(Q_{t_0}\cap A)]\\
&&\quad\;\,+ \,[\FF^c (Q_{t_0}\cap A) - \FF^c (Q_t\cap A)] \stackrel{\eqref{m4},\eqref{e:oscil}}{<}
 -\frac{\eps}{2}.
\end{eqnarray*}
\end{itemize}
The previous estimates 
imply properties $(b)$ and $(c)$ of the lemma. This concludes the proof.
\end{proof}

\subsection{Replacements} We use the notion of replacements introduced by Pitts, but we also require a density one condition, as in \cite{ZZ}.

\begin{Definition}
Let $V\in\cV^\infty_c$ and $U\subset M$ be 
an open set. We say that a varifold $V'\in \cV^c_\infty$ is a 
replacement for $V$ in $U$ if $V'\res (M\setminus \bar{U})=V\res M\setminus \bar{U}$, $\theta(x,V')=1$ for $\|V'\|$-a.e. $x \in U$, and $V\res U$ is a smooth $c$-stable surface. 
\end{Definition}

We aim to show that almost minimizing varifolds posses replacements.

\begin{Proposition}\label{rep}
Consider $\{\Om^j\}$, $V$ and $r:M\to (0,
\infty]$ be as in Proposition
\ref{p:almost1}. For every $x\in M$ and $An\in \An_{r(x)} (x)$ there
exist a varifold $\tilde{V}$, a min-max sequence $\{\tilde{\Om}^j\}$
and a map $r':M\to (0,\infty]$ satisfying the following properties:
\begin{itemize}
\item $\tilde{V}$ is a replacement for $V$ in $An$
and $\partial \tilde{\Om}^j\rightharpoonup \tilde{V}$ in the varifolds topology;
\item $\tilde{\Om}^j$ is a.m. in every $An'\in \An_{r'(y)} (y)$
for every $y\in M$;
\item $r'(x)=r(x)$;
\item $\lim_{j\to \infty} (\FF^c (\tilde{\Om}^j)-\FF^c (\Om^j))= 0.$
\end{itemize}
\end{Proposition}
\begin{Remark}
Once we have applied Proposition \ref{rep} obtaining the replacement $\tilde{V}$, thanks to the properties above we can apply Proposition \ref{rep} again to  $\{\tilde{\Om}^j\}$, $\tilde{V}$, and $r'$ in any $An'\in \An_{r (x)} (x)\cup\bigcup_{y\neq x}
\An_{r'(y)} (y)$.
\end{Remark}

\subsection{Proof of Proposition \ref{rep}}\label{ss:rep_setting}
The proof of Proposition \ref{rep} follows the proof of \cite[Proposition 2.6]{DLTas}, however it requires non-trivial adaptations. We fix $An\in \An_{r(x)} (x)$, and for every $j$, consider
the class $\mathcal{E} (\Omega^j, An)$
of sets $Q$ such that 
there is a family $\{\Omega_t\}$ satisfying
$\Omega_0=\Omega^j$,
$\Omega_1=Q$,
\eqref{m1}, \eqref{m2} and \eqref{m3} for $\eps=\frac{1}{j}$ and $U=An$.
We choose a minimizing sequence $\{\Omega^{j,k}\}_k$
 for $\FF^c$ in the class $\mathcal{E}
(\Omega^j, An)$ such that $\Omega^{j,k}$ converges to a Caccioppoli set
$\tilde{\Omega}^j$, $\partial \Om^{j,k}$ converges to a varifold $V^j\in \cV^c_\infty$, and both $V^j$ and a diagonal sequence $\partial \tilde{\Om}^j = \partial \Om^{j, k(j)}$ converge to the same varifold $\tilde{V}$.

To prove Proposition \ref{rep} we need four intermediate lemmas:

\begin{Lemma}\label{l:concon} For every $j\in \NN$ and
every $y\in An$ there exist a ball $B_\rho (y)\subset An$ 
and $k_0\in \NN$ for which the following statement holds:

For every $k\geq k_0$ and any open set $Q$ satisfying:
\begin{itemize}
\item[-] $\partial Q$ is smooth in the complement of a finite set,
\item[-] $Q\setminus B_\rho (y) = \Omega^{j,k}\setminus B_\rho (y)$,
\item[-] and $\FF^c(Q) \leq \FF^c(\Omega^{j,k})$,     
\end{itemize}
then $Q\in \mathcal{E} (\Omega^j, An)$.
\end{Lemma}

\begin{Lemma}\label{l:min}
$\partial \tilde{\Omega}^j\cap An$ is a smooth c-stable surface
in $An$ and $\partial \tilde{\Omega}^j\res An = V^j \res An$.
\end{Lemma}

\begin{Lemma}\label{l:concon1} 
Assume that there exists $y\in \partial_+ An$ and a ball $B_{\rho_y} (y)$ such that $\partial \Omega^j\cap B_\rho (y)\cap An=\gamma$ where $\gamma$ is a smooth simple curve. There exist a ball $B_{\rho} (y)\subset B_{\rho_y} (y)$ and $k_0\in \NN$  for which the following statement holds:

For every $k\geq k_0$ and any open set $Q$ satisfying:
\begin{itemize}
\item[-] $\partial Q$ is smooth in the complement of a finite set,
\item[-] $Q\setminus (B_\rho (y)\setminus An) = \Omega^{j,k}\setminus (B_\rho (y)\setminus An)$,
\item[-] and $\FF^c(Q) \leq \FF^c(\Omega^{j,k})$,     
\end{itemize}
then $Q\in \mathcal{E} (\Omega^j, An)$.
\end{Lemma}

\begin{Lemma}\label{l:min1}
Fix $y\in \partial_+ An$ and assume $\partial \Omega^j$ is a smooth embedded c-stable surface in a ball $B_{\rho_y} (y)$. Then there exists a ball $B_{\rho} (y)$ such that $\partial \tilde{\Omega}^j\cap B_\rho (y)\cap An$ is smooth up to the boundary $\partial \Omega^j\cap \partial_+ An \cap B_\rho (y)$. 
\end{Lemma}

Assuming the validity of the intermediate lemmas \ref{l:concon}, \ref{l:min}, \ref{l:concon1}, \ref{l:min1}, we can conclude the proof of Proposition \ref{rep} as follows:

By definition $\tilde{\Om}^j$ coincides with
$\Om^j$ outside $An$.
Fix an annulus $An'=An (x, \eps, r(x)-\eps)\supset\supset
An$. By the assumptions on
$\{\Om^j\}$, then $\{\tilde{\Om}^j\}$ is also
a.m. in $An'$.
As $\tilde{\Om}^j$ is a.m. in every open subset $U\subset An'$, by the arbitrarity of $\eps$ we deduce that $\tilde{\Om}^j$ is a.m. in any
annulus in $\An_{r(x)} (x)$. Moreover we have that $M = An'\cup (M\setminus An)$.

For every $y\in M\setminus An$, $y\neq x$, 
we define $r'(y):=\min\{r (y), \d (y, An)\}$. If
$An''\in \An_{r'(y)} (y)$, then $\Om^j\cap An'' =
\tilde{\Om}^j \cap An''$, which in turn implies that
$\{\tilde{\Om}^j\}$ is a.m. in $An''$. If
$y\in An$, then we define $r'(y) := \min\{r(y), \d (y,
\partial An')\}$. If $An''\in \An_{r'(y)} (y)$,
then $An''\subset An'$ and, since $\{\tilde{\Om}^j\}$ is a.m.
in $An'$, then $\{\tilde{\Om}^j\}$ is a.m.
in $An''$. 

Now we prove that $\tilde{V}$ is a replacement
for $V$ in $An$. Theorem \ref{T:compact} implies that $\tilde{V}$
is a $c$-stable surface in $An$. We need to show that $\tilde{V}\in \cV^c_\infty$.
To this aim, we consider a partition
of unity $\{\psi_1, \psi_2\}$ for the covering $\{An', M\setminus An\}$ of $M$. Since $\tilde{V}$ coincides with $V$  in $M\setminus An$, for every $X\in \mathcal{X}(M)$ we compute
$$|[\delta_\FF \tilde{V}] (X)| \leq |[\delta_\FF \tilde{V}] (\psi_1 X)|
+| [\delta_\FF \tilde{V}] (\psi_2X) |= [\delta_\FF \tilde{V}] (\psi_1 X)+c\int \psi_2 |X|\,d\|\tilde{V}\|.$$
Hence, it is enough to prove that 
$$|\delta_{\FF} \tilde{V}(X)| \le c\int  |X|\,d\| \tilde{V}\|\text{ for all } X \in \mathcal{X}_c (An').$$
If the inequality above does not holds, then there exists $X\in \mathcal{X}_c (An')$ such that 
$$\delta_{\FF} \tilde{V}(X) +c\int  |X|\,d\| \tilde{V}\leq -4C <0.$$
We consider the isotopy $\Phi$ generated by $X$ through the ODE
$\frac{\partial\Phi(x,t)}{\partial t}=X(\Phi(x,t))$. Denoting 
$$
 \tilde{V} (t)\;:=\;\Phi(t)_{\#} \tilde{V}\qquad
\Sigma^j (t)\;=\;\Phi(t,\partial \tilde{\Om}^j)
\qquad
\tilde\Omega^j (t)\;=\;\Phi(t, \tilde{\Om}^j),
$$
there exists $\eps>0$ such that
$\delta_{\FF} \tilde{V}(t)(X) +c\int  |X|\,d\| \tilde{V}(t)\|\leq-2C$ for every $t\leq \eps$.
Since $\Sigma^j(t)\rightharpoonup \tilde{V}(t)$
in the sense of varifolds, there exists $N\in \NN$ such that 
\begin{equation}\label{ciaociao}
\delta_{\FF} \Sigma^j (t)(X) +c\int  |X|\,d\| \Sigma^j (t)\|\leq 
-{C}\quad\text{for every}\;j>N\;\text{and}\; t\leq\eps.
\end{equation}
Integrating \eqref{ciaociao} in $t$, we conclude that
$ \FF^c(\tilde\Omega^j (t))\leq \FF^c (\tilde{\Omega}^j) - Ct$
for every $t\in [0,\eps]$ and $j\geq N$, which
 contradicts the a.m. property
of $\tilde{\Omega}^j$ in $An'$. This conclude the proof that $\tilde{V}$ is a replacement for $V$ in $An$.

Finally, observe that $\FF^c(\tilde{\Om}^j)
\leq \FF^c (\Om^j)$ by construction and
$\liminf_n (\FF^c (\tilde{\Om}^j)-\FF^c (\Om^j))\geq 0$,
because otherwise we would contradict the a.m. property
of $\{\Om^j\}$ in $An$. We thus conclude that:
$$\lim_{n\to \infty} (\FF^c (\tilde{\Om}^j)-\FF^c (\Om^j))= 0.$$

\subsection{Proof of the intermediate lemmas \ref{l:concon}, \ref{l:min}, \ref{l:concon1}, \ref{l:min1}}
\label{ss:concon}

\begin{proof}[Proof of Lemma \ref{l:concon}]
The proof of Lemma \ref{l:concon} is analogous to the proof of \cite[Lemma 4.1]{DLTas}. Hence we refer the reader to \cite{DLTas} for more detailed justifications of the proof.

Let us fix $j\in \NN$  and 
$y\in An$. 
Let $\rho>0$ be such that $B_{2\rho} (y) \subset An$ 
and consider
an open set $Q$ satisfying the properties in the statement of the Lemma.
Given the local nature of Lemma \ref{l:concon}, by Remark \ref{metricinG} we will assume without loss of generality that the ambient space is $\R^3$.

As in the Step 1 of the proof of \cite[Lemma 4.1]{DLTas}, we can choose $r\in (\rho, 2\rho)$ 
such that, for every $k$, $\partial \Om^{j,k}$ is
regular in a neighborhood of $\partial B_r (y)$ and intersects
it transversally.
For each $z\in \overline{B}_r (y)$ 
we denote with $[y,z]$ the closed segment with end points $y$ and $z$, and $(y,z):=[y,z]\setminus \{y,z\}$.
We define $T$ as the following open cone
$$
T \;=\; \bigcup_{z\in \partial B_r(y) \cap \Omega^{j,k}} (y,z)\, .
$$

We consider $\eps>0$ and a smooth function $\psi:[0, 2\rho]
\to [0,2\rho]$, such that: $|\psi(s) -s|\leq\eps$ and $0\leq \psi'(s)\leq 2$ for every $s$; $\psi (s)=s$ if $|s-r|>\eps$; and $\psi\equiv r$ in a neighborhood of $r$.

We set $\Psi (t,s) := (1-t) s + t \psi (s)$ and, for every $\lambda\in [0,1]$ and every $z\in \overline{B}_r (y)$, we define $\tau_\lambda (z)\in [y,z]$
such that $\d(y,\tau_\lambda (z))=\lambda\, 
\d(y,z)$. For $1<\lambda$, we define $\tau_\lambda(z)$ to be the corresponding point on the segment that is the extension of $[y,z]$.  We can finally define an homotopy $\tilde{\Omega}_t$ between $\Omega^{j,k}$  and $\tilde{\Omega}_1$ as follows:
\begin{itemize}
\item $\tilde{\Omega}_t\setminus An (y, r-\eps, r+\eps)
= \Omega^{j,k} \setminus An (y, r-\eps, r+\eps)$;
\item $\tilde{\Omega}_t \cap \partial B_s (y)
= \tau_{s/\Psi (t,s)} (\Omega^{j,k}\cap \partial B_{\Psi (t,s)})$
for every $s\in (r-\eps, r+\eps) $.
\end{itemize}
Since $\partial \Om^{j,k}$ is
regular in a neighborhood of $\partial B_r (y)$ and intersects
it transversally, we can choose $\eps$ to have $\max_t \FF^c (\partial \tilde{\Omega}_t) - 
\FF^c (\partial \Omega^{j,k})$ 
 as small as desired. Moreover $\tilde{\Omega}_1$ coincides with $T$ in a neighborhood
of $\partial B_r (y)$.
Since $Q$ coincides with $\Omega^{j,k}$
on $M\setminus B_\rho (y)$, the same argument 
can be applied to $Q$. We deduce that one can assume $T=Q=\Omega^{j,k}$ in a neighborhood of $\partial B_r (y)$.

We now consider the following family of open
sets $\{\Omega_t\}_{t\in [0,1]}$:
\begin{itemize}
\item $\Omega_t\setminus \overline{B}_r (y) = \Omega^{j,k}
\setminus \overline{B}_r (y)$ for every $t$;
\item $\Omega_t \cap An (y, |1-2t| r, r)
= T\cap An (y, |1-2t| r, r)$ for every $t$;
\item $\Omega_t \cap \overline{B}_{(1-2t)r} (y)
= \tau_{1-2t} (\Omega^{j,k}\cap \overline{B}_r (y))$
for $t\in [0,\frac{1}{2}]$;
\item $\Omega_t\cap \overline{B}_{(2t-1)r} (y)
= \tau_{2t-1} (Q\cap \overline{B}_r (y))$
for $t\in [\frac{1}{2}, 1]$.
\end{itemize}

This is a generalized smooth family and it satisfies (so2)
of Definition \ref{d:sweep}. It remains to check that
\begin{equation}\label{e:am3bis}
\max_t \FF^c (\Omega_t)
\;\leq\; \FF^c (\Omega^{j,k}) +\frac{1}{8j}\, 
\qquad \forall k\geq k_0.
\end{equation}

We observe that for every $r<2\rho$ and $\gamma\in [0,1]$:
\begin{equation}\label{uuuuu1}
\FF(\partial T\cap B_r(y))\leq \lambda \cH^2 (\partial T\cap B_r(y)) \;\leq\; \lambda r \cH^{1} (\partial
\Omega^{j,k}\cap \partial B_r (y))
\end{equation}
\begin{equation}\label{uuuuu2}
\begin{split}
\FF([\partial (\tau_\gamma (\Omega^{j,k}
\cap \overline{B}_r (y)))]\cap B_{\gamma r} (y))&\;\leq \;\lambda \cH^2 ([\partial (\tau_\gamma (\Omega^{j,k}\cap \overline{B}_r (y)))]\cap B_{\gamma r} (y))\\
&\; \leq \;
\lambda \cH^2 (\partial \Omega^{j,k} \cap B_r (y)) \leq 
\lambda^2 \FF (\partial \Omega^{j,k} \cap B_r (y))
\end{split}
\end{equation}
\begin{equation}\label{e:shrink2}
\begin{split}
\FF ([\partial (\tau_\gamma (Q
\cap \overline{B}_r (y)))]\cap B_{\gamma r} (y))&\;\leq \; \lambda\cH^2 ([\partial (\tau_\gamma (Q
\cap \overline{B}_r (y)))]\cap B_{\gamma r} (y))\\
&\; \leq\; \lambda \cH^2 (\partial Q \cap B_r (y)) \leq \lambda^2 \FF (\partial Q \cap B_r (y))
\end{split}
\end{equation}
\begin{equation}\label{are}
\int_0^{2\rho} \cH^{1} (\partial \Omega^{j,k}
\cap \partial B_\tau (y))\, d\tau
\;\leq\;  \cH^2 (\partial \Omega^{j,k}\cap B_{2\rho} (y))\; \leq\; \lambda \FF (\partial \Omega^{j,k}\cap B_{2\rho} (y)).
\end{equation}
Since $\FF^c(Q) \leq \FF^c(\Omega^{j,k})$, we deduce from \eqref{e:shrink2} that
\begin{equation}\label{e:shrink22}
\begin{split}
\FF &([\partial (\tau_\gamma (Q
\cap \overline{B}_r (y)))]\cap B_{\gamma r} (y))\;\leq \;\lambda^2 \FF (\partial Q \cap B_r (y))\\
&\;\leq \; \lambda^2 [\FF (\partial \Omega^{j,k} \cap B_{2\rho} (y))+ c\haus^3 (Q \cap B_{2\rho} (y)) - c\haus^3 (\Omega^{j,k}\cap B_{2\rho} (y))]\\
&\;\leq \; \lambda^2 [\FF (\partial \Omega^{j,k} \cap B_{2\rho}(y))+c\haus^3 (B_{2\rho} (y))]\leq  \lambda^2 [\FF (\partial \Omega^{j,k} \cap B_{2\rho} (y))+ C \rho^3],
\end{split}
\end{equation}
where $C$ is a (possibly changing) constant depending just on $F$ and $c$.
Hence \eqref{uuuuu1}, \eqref{uuuuu2} and \eqref{e:shrink22}
imply that
\begin{equation*}
\max_t \FF(\partial \Omega_t)-\FF(\partial \Omega^{j,k})
\;\leq\; \lambda^2\FF (\partial \Omega^{j,k}\cap B_{2\rho} (y))
+ \lambda r \cH^{1} (\partial \Omega^{j,k}\cap \partial B_r (y))+  C \rho^3\, .
\end{equation*}
Moreover, by \eqref{are} we can
choose $r\in (\rho, 2\rho)$ which
satisfies:
$$
\cH^{1} (\partial \Omega^{j,k}\cap \partial
B_r (y))\;\leq\;\frac{2\lambda}{\rho} \FF(\partial \Omega^{j,k} \cap B_{2\rho} (y))\, .
$$
Hence, we conclude
\begin{equation}\label{x2}
\max_t \FF (\partial \Omega_t)
\;\leq\; \FF (\partial \Omega^{j,k})
+ 2\lambda^2  \FF (\partial \Omega^{j,k}\cap B_{2\rho} (y))+ C\rho^3\, .
\end{equation}
Furthermore, there exists $k_0$ such
that 
\begin{equation}\label{re3}
\FF (\partial \Omega^{j,k}\cap B_{2\rho} (y))\;\leq\; 2 \FF(V^j \res B_{4\rho} (y))
\;\leq\; 2 \Lambda \|V^j\| (B_{4\rho} (y)) \qquad \mbox{for every $k\geq k_0$.}
\end{equation}
By the non concentration Lemma \ref{nonconc},
\begin{equation}\label{x3}
\|V^j\| (B_{4\rho} (y))\leq C \|V^j\| (M) \alpha^{\log_2(\rho_0/\rho)}\, .
\end{equation}
We recall that $\alpha\in (0,1)$ is the constant of Lemma \ref{nonconc} and it depends just on $F$.
To conclude, we observe that 
\begin{equation}\label{x4}
c\haus^3 (\partial \Omega^{j,k}) - \min_t c\haus^3 (\partial \Omega_t) \leq c\haus^3 (B_{2\rho} (y))\leq  C \rho^3\, .
\end{equation}
Combining \eqref{x2}, \eqref{re3}, \eqref{x3} and \eqref{x4}, we deduce that
\begin{equation}\label{x5}
\max_t \FF^c (\partial \Omega_t)
\;\leq\; \FF^c (\partial \Omega^{j,k})
+ C \|V^j\| (M) \alpha^{\log_2(\rho_0/\rho)} + C \rho^3\, .
\end{equation}
As $\alpha \in (0,1)$, choosing $\rho<\rho_0$ small enough, by \eqref{x5} we conclude \eqref{e:am3bis}, as desired.
\end{proof}

\begin{proof}[Proof of Lemma \ref{l:min}]
Fix $j\in \NN$ and $y\in An$ and
let $B_\rho (y)\subset An$ be the ball given by Lemma
\ref{l:concon}. Repeating verbatim the proof of \cite[Lemma 4.2]{DLTas}, and replacing the use of \cite[Lemma 4.1]{DLTas} with Lemma \ref{l:concon}, one can prove that that $\tilde{\Omega}^j$
minimizes $\FF^c$ in the class $\mathcal{P}
(\tilde{\Omega}^j, B_{\rho/2} (y))$ of the finite perimeter sets equal to $\tilde{\Omega}^j$ outside of $B_{\rho/2} (y)$.

Now, we want to show that  $\partial \tilde{\Omega}^j\res An = V^j \res An$. First we prove that 
\begin{equation}\label{pp}
\lim_{k\to \infty}\FF (\partial \Omega^{j,k})=\FF(\partial \tilde{\Omega}^j).
\end{equation}
Indeed, if this is not the case, then we would
have
$$
\FF^c (\tilde{\Omega}^j \cap B_{\rho/2} (y))
\;<\; \limsup_{k\to \infty} 
\FF^c (\Omega^{j,k}\cap B_{\rho/2} (y))
$$
for some $y\in An$ and some $\rho$ to which we can
apply the conclusion Lemma \ref{l:concon}.
We can then use $\tilde{\Omega}^j$ in place
of $Q$ in the argument of the previous step
to contradict, once again, the minimality of the
sequence $\{\Omega^{j,k}\}_k$.
We apply \cite[Theorem 1, Section 3.4]{CartCurr} to \eqref{pp} to deduce that
\begin{equation}\label{ppp}
\lim_{k\to \infty}\H^2(\partial \Omega^{j,k})=\H^2(\partial \tilde{\Omega}^j).
\end{equation}
Applying \cite[Proposition A.1]{DLTas} to \eqref{ppp}, we deduce that  $\partial \tilde{\Omega}^j\res An = V^j \res An$.

To conclude, we prove the $c$-stability of the surface
$\partial \tilde{\Omega}^j$. Assume by contradiction that there exists a smooth volume preserving vector field $X$  compactly supported in $An$ such that $\delta^2_{\Fc} \tilde{\Omega}^j(X,X)=-C<0$. 
There exists a map $\varphi:t\in \R \mapsto \varphi_t\in C^\infty (An,An)$ solving the following ODE:
\begin{equation*}
\begin{cases}\frac{\partial \varphi_t(x)}{\partial t}=X(\varphi_t(x))&\quad \forall x\in M,\\
\varphi_0(x)=x &\quad \forall x\in M.
\end{cases}
\end{equation*}
We set
$$ \tilde{\Omega}^j_t:=(\varphi_t)(\tilde{\Omega}^j),\qquad \mbox{and} \qquad  \Omega^{j,k}_t:=(\varphi_t)(\Omega^{j,k}).$$
Using the information on the first and second variation, we get that
\begin{equation}
\label{integro1}\Fc(\tilde{\Omega}^j_\e)\leq \Fc(\tilde{\Omega}^j) - C\e^2.
\end{equation}
We can easily define an homotopy $\psi(t,x)=\varphi_{\e t}(x)$ for every $(t,x)\in [0,1]\times An$ and $\psi(t,x)=x$ for every $(t,x)\in [0,1]\times (M\setminus An)$, so that $\tilde{\Omega}^j_\e:=\psi(1,\tilde{\Omega}^j)$. This contradicts the minimality of the sequence $\Omega^{j,k}$ in $\mathcal{E}
(\Omega^j, An)$.
\end{proof}

\begin{proof}[Proof of Lemma \ref{l:concon1}]
The proof is very similar to the one of Lemma \ref{l:concon}.
It is achieved by exhibiting a suitable
homotopy between $\Omega^{j,k}$ and $Q$. 
For simplicity, we will assume the ambient space is $\R^3$, without loss of generality (as this is a local result). Up to translation and rotation, we can assume $y=0$, that the tangent space to $\partial_+ An$ in $0$ is $(e_1,e_2)$, that $\dot \gamma/|\dot \gamma|=e_1$, and $An\subset \{\langle x, e_3\rangle >0\}$. Up to consider $\rho$ small enough, for every $a\in [-\rho,\rho]$ there exists a unique $y_a \in \gamma\cap \{x_1=a\}$.
 The key idea is:
\begin{itemize}
\item[-] First deform $\Omega^{j,k}$ to
the set $\tilde{\Omega}$ which is
the union of $\Omega^{j,k}\setminus (B_\rho (y)\setminus An)$ and the family of 1-dimensional cones $\{C_a\}_{a\in [-\rho,\rho]}$ lying in $\{x_1=a\}$, with vertices $y_a$ and base $\Omega^{j,k}\cap \partial B_\rho (y)\cap An$;
\item[-] Then deform $\tilde{\Omega}$ to $Q$ in a similar way.
\end{itemize}
Using this scheme, the proof becomes an immediate modification of the proof of Lemma \ref{l:concon} and we will omit it.
\end{proof}

\begin{proof}[Proof of Lemma \ref{l:min1}]
We observe that $\partial \Omega^j$ satisfies the assumptions of Lemma \ref{l:concon1} in $B_{\rho_y} (y)$. Let $B_\rho (y)$ be the ball given by Lemma \ref{l:concon1}. We claim that $\tilde{\Omega}^j$ minimizes $\FF^c$ in the class $\mathcal{P}
(\tilde{\Omega}^j, B_{\rho/2} (y)\cap An)$ of the finite perimeter sets equal to $\tilde{\Omega}^j$ outside of $B_{\rho/2} (y)\cap An$. If this was true, by Theorem \ref{boundregt} in Appendix \ref{boundreg}, we would deduce that $\partial \tilde{\Omega}^j\cap B_{\rho/2} (y)\cap An$ is smooth up to the boundary $ \partial \tilde{\Omega}^j\cap \partial_+ An \cap B_{\rho/2} (y)$. 

Let us prove the claim.
Assume, by contradiction, that $Q$ is
a Caccioppoli set with $Q\setminus (B_{\rho/2} (y)\cap An)
= \tilde{\Omega}^j\setminus (B_{\rho/2} (y)\cap An)$ and
\begin{equation}\label{e:competitor}
\FF^c (Q) \;<\; \FF^c (\tilde{\Omega}^j)
-2\eta\, .
\end{equation}
Then there exists $\bar \delta>0$ such that, denoting with $An_{t}:=(1-t)An$, 
\begin{equation}\label{e:competitor1}
\FF^c (Q\cap B_{\rho/2} (y)\cap An_{t}) \;<\; \FF^c (\tilde{\Omega}^j\cap B_{\rho/2} (y)\cap An_{t})-\eta, \qquad \forall t\in (0,\bar \delta)
\, .
\end{equation}
Since the convergence $\ind_{\Omega^{j,k}} \to 
\ind_{\tilde{\Omega}^j}$ is strong in $L^1$, 
there exist
 $\tau\in (\rho/2, \rho)$ and $\delta \in (0,\bar \delta)$ such that (up to subsequences)
\begin{equation}\label{uuu}
\lim_{k\to\infty}
\|\ind_{\tilde{\Omega}^j} - \ind_{\Omega^{j,k}}\|_{L^1
(\partial B_\tau (y)\cup \partial_+  An_\delta)}\;=\; 0\,,
\end{equation}
\begin{equation}\label{uuu1}
\FF^c (\tilde{\Omega}^j\cap B_\tau (y)\cap An_\delta)
\leq \liminf_{k\to\infty}
\FF^c (\Omega^{j,k}\cap B_\tau (y)\cap An_\delta),\quad \mbox{and} \quad \FF^c (\tilde{\Omega}^j)
\leq \liminf_{k\to\infty}
\FF^c (\Omega^{j,k})\, .
\end{equation}
We define 
$$
Q^{j,k}\;=\; (Q \cap B_\tau (y)\cap An_\delta)
\cup (\Omega^{j,k}\setminus (B_\tau (y)\cap An_\delta))\, .
$$
Hence,  we deduce from \eqref{e:competitor1}, \eqref{uuu}
and \eqref{uuu1} that
$$
\limsup_{k\to\infty} [\FF^c (Q^{j,k})
- \FF^c (\Omega^{j,k})]
\;\leq\; -\eta\, .
$$
We fix now $k$ and a compactly supported
kernel $\varphi$, then we define $g_\eps := \ind_{Q^{j,k}}*\varphi_\eps$
and the level set $\Xi_\eps
:= \{g_\eps > t\}$ for some $t\in (\frac{1}{4},\frac{3}{4})$. 
Then $\FF^c (\Xi_\eps)\to\FF^c (Q^{j,k})$ as $\eps\to 0$, see \cite{Giu,MP}.

As $\Xi_\eps$ does not coincide with $\Omega^{j,k}$ outside $B_\rho (y)\setminus An$, we need to modify $\Xi_\eps$.
Hence we choose $(a,b)\subset (\tau, \rho)$ and $(\alpha,\beta)\subset (0,\delta)$ so that $\Sigma^{j,k} := \partial \Omega^{j,k}\cap ((\overline{B_b (y)}
\setminus B_a (y))\cup (\overline{An_\alpha}
\setminus An_\beta))$ is smooth. We also consider
a smooth tubular neighborhood $T$ of $\Sigma^{j,k}$ and the associated
normal coordinates $(\xi,\sigma)$.
As
$Q^{j,k} \setminus (B_\tau (y)\cap An_\delta)
= \Omega^{j,k} \setminus (B_\tau (y)\cap An_\delta)$, then there exists $\eps>0$ and a smooth function $f_\eps:\xi\to f_\eps(\xi)$ such that
$\partial \Xi_\eps\cap ((\overline{B_b (y)}
\setminus B_a (y))\cup (\overline{An_\alpha}
\setminus An_\beta))\subset T$ and $T\cap \Xi_\eps:=\{\sigma<f_\eps (\xi)\}$, and $f_\eps \to 0$ smoothly as $\eps\to 0$, 

Arguing as in the proof of Lemma \ref{l} one can modify $Q^{j,k}$
to a set $\Xi^{j,k}$ satisfying:
\begin{itemize}
\item $\partial \Xi^{j,k}$ is smooth in the complement of a finite set;
\item $\Xi^{j,k}\setminus (B_\rho (y)\setminus An) = \Omega^{j,k}\setminus (B_\rho (y)\setminus An) $;
\item $\limsup_k (\FF^c (\Xi^{j,k}) - \FF^c
(\Omega^{j,k}))\leq -\eta<0$.
\end{itemize} 
Lemma \ref{l:concon1}
implies that $Q^{j,k}\in \mathcal{E}(\Omega^j, An)$ for $k$ sufficiently large,
contradicting the minimality of the sequence
$\Omega^{j,k}$.
\end{proof}

\subsection{Regularity}\label{regsec}
In this section we use Proposition \ref{rep} to get the desired regularity:
\begin{Proposition}\label{p:reg}
Let $V$ be as in Proposition \ref{p:almost1} (and consequently as in Proposition \ref{rep}).
Then $V$ is induced by a non trivial surface $\Sigma$ with multiplicity one, which is smooth, $c$-stable and almost embedded outside of one point $p\in M$.
 \end{Proposition}
 We remark that Proposition \ref{p:reg} implies the validity of Theorem \ref{t:main}.
In order to prove Proposition \ref{p:reg}, we will use some intermediate results that are of independent interest. Since they are also valid in general dimension, we state and prove them in Section \ref{br}.

As a first step towards the proof of Proposition~\ref{p:reg} we have the following:

\begin{Lemma}\label{rect} Let $V\in \cV^\infty_c$. If there exists a positive function $r$ on $M$ such that $V$ has  a replacement in any annulus $\an\in \An_{r(x)} (x)$, then $V$ is integral. Moreover, there exists $C>0$, depending only on $F$ and $c$, such that $\theta_* (x, V)\geq C$ for any $x\in \supp (\|V\|)$.
\end{Lemma}

\begin{proof} 
Fix $x\in \supp (\|V\|)$. Given the local nature around $x$ of this lemma, by Remark \ref{metricinG} we will assume without loss of generality that the ambient space is $\R^3$. Fix $r<\frac 1{\sqrt{\Delta}}$, where $\Delta$ is the constant defined in defined in \cite[page 294, Section 1.4]{Allard1983} and depends just on $F$, and $2r<\lambda/c$. Replace $V$ with $V'$ in $\an (x, r, 2r)$. We claim that $\|V'\|$ cannot be identically $0$ on $\an (x,r,2r)$. Assume it was; since $x\in \supp (\|V'\|)$, there would be a $\rho\leq r$ such that $V'$ ``touches'' $\partial B_\rho$ from the interior. More precisely, there would exist $\rho$ and $\e$ such that $\supp \|V'\|\cap \partial B_\rho (x)\neq \emptyset$ and $\supp \|V'\|\cap \an (x, \rho, \rho+\e)=\emptyset$. Moreover $\supp \|V'\res \an (x, \rho, \rho+\e)\|$ is a $(2,c)$ subset of $\an (x, \rho, \rho+\e)$, as proved in \cite[Lemma 4.5]{DKS}. In particular, we deduce from \cite[Proposition 2 (iii)]{DDH} that $H_F^{\partial B_\rho (x)}(p)\leq c$ for every $p\in \supp \|V'\|\cap \partial B_\rho (x)\neq \emptyset$. On the other hand, since $\rho<2r<\lambda /c$, we deduce that $H_F^{\partial B_\rho (x)}> \frac \lambda\rho>c$, which is the desired contradiciton. 
\begin{equation}\label{arg}
\mbox{Thus $V'\res \an (x,r,2r)$ is a non--empty smooth $c$-stable surface $\Sigma$. }
\end{equation}
We fix $y\in \partial B_{3r/2}(x)\cap \Sigma$, which exists because otherwise we could apply the same argument as before to the $c$-stable surface $\Sigma$ in $B_{3r/2}(x)$, contradicting the maximum principle. We apply \cite[page 305, Theorem, Equation (9)]{Allard1983} with $Z=\an (x,r,2r)$ and $M=\Sigma$: the whole set of assumption needed is in \cite[page 294, Section 1.4]{Allard1983} and are trivially satisfied. Since $r<\frac 1{\sqrt{\Delta}}$, we deduce that there exists $C>0$ depending just on $F$ and $c$, such that
$$\haus^2(\Sigma\cap B_{r/2}(y))\geq C(r/2)^2.$$
In particular
$$\frac{\|V\|(B_{2r}(x))}{(2r)^2}\geq \frac{\haus^2(\Sigma\cap B_{r/2}(y))}{(2r)^2}\geq \frac{C(r/2)^2}{(2r)^2}\geq \frac C{16}.$$
Passing to the liminf in $r$, by the arbitrarity of $x$, we deduce that there exists $C>0$, depending only on $F$ and $c$ such that $\theta_* (x, V)\geq C$ for any $x\in \supp (\|V\|)$, as claimed.
The density lower bound, together with $V$ having bounded anisotropic first variation, implies that $V$ is a rectifiable varifold, see \cite[Theorem 1.2]{DDG2} and Remark \ref{rectif}.

We next prove that $V$ is integral. Fix $x\in \supp (\|V\|)$ such that the tangent cone $\C$ is unique and is a plane $\theta \pi$ with constant multiplicity $\theta$. We recall that this is true for $\|V\|$-a.e. $x$, given that $V$ is rectifiable, cf. \cite{Si}. 
We need to prove that $\theta$ is an integer value.

We observe that $\delta_\FF(C)=0$. Consider a sequence $\rho_n\downarrow 0$ such that $(\eta_{x,\rho_n})_\# V \rightharpoonup \C$. Replace $V$ by $V'_n$ in $\an (x, \rho_n/4, 3\rho_n/4)$ and set $W'_n= (\eta_{x,\rho_n})_\# V'_n$. After possibly passing
to a subsequence, we can assume that $W'_n  \rightharpoonup \C'$, where $\C'$ satisfies again $\delta_\FF(\C')=0$. A trivial consequences of the definition of replacements is that
\begin{equation}\label{CC'}
\mbox{$\C'=\C=\theta \pi$ in $\B_{1/4}\cup \aan (0, 3/4, 1)$.}
\end{equation}
Moreover, by definition of replacement and rescaling properties, $W'_n$ are smooth, almost embedded, $(c\rho_n)$-stable surfaces in $\aan (0, 1/4, 3/4)$. Hence, by Theorem \ref{T:compact}(ii) we conclude that $W'_n$ converge locally smoothly (with integer multiplicity) to some smooth embedded stable surface $\Sigma'$ in $\aan (0, 1/4, 3/4)$. 
Since $\delta_\FF(\C')=0$, with the same argument used on $V'$ to get \eqref{arg}, we can show that
\begin{equation}\label{cond}
\mbox{cl}(\Sigma')\cap \partial \B_{3/4} \subset \pi.
\end{equation}
 We claim that $\Sigma'\subset \pi$. Indeed,  assume by contradiction the claim is false. Up to rotation, we can assume $\pi=(e_1,e_2)$. Moreover, without loss of generality we can assume $\Sigma' \cap \{\langle x, e_3\rangle >0\}\neq \emptyset$ (the case where $\Sigma'$ lies below $\pi$ is analogous). 
There exists $\min\{a>0: \Sigma' \cap \{\langle x, e_3\rangle =a\}\neq\emptyset\}$. By the classical maximum principle (cf. for instance [35, Corollary 5.1]):
$$\{\langle x, e_3\rangle =a\}\cap \aan (0, 1/4, 3/4) \subset \Sigma'\cap \aan (0, 1/4, 3/4),$$
 which contradicts \eqref{cond}.
Recalling \eqref{CC'}, this implies that $\supp(\|\C'\|) \subset \pi$.  Since $\delta_\FF(\C')=0$, the Constancy Theorem \cite[Proposition 5]{DDH}, together with \eqref{CC'}, implies that  $\C'=\theta \pi$. Since $\Sigma'$ has integer multiplicity in $\aan (0, 1/4, 3/4)$, we conclude that $\theta$ is integer.
\end{proof}

\begin{proof}[Proof of Proposition~\ref{p:reg}] 
Fix $x\in M$, and consider 
\begin{equation}\label{rhowow}
2\rho:=\frac 12\min\left \{\Inj(M),\lambda/c, r(x),\frac{1}{\lambda(c+\lambda+4\lambda^3)}\right \},
\end{equation}
where $r(x)$ is as in Proposition \ref{p:almost1}.  Consider a replacement $V'$ for $V$ in $\an(x, \rho, 2\rho)$, and let $\Sigma'$ be the $c$-stable surface given by $V'$ in $\an(x, \rho,2\rho)$. Thanks to Lemma \ref{stimatouchingset}, we can choose $t\in ]\rho, 2\rho[$ such that 
\begin{equation}\label{isolato}
\haus^0(\mathcal S(\Sigma')\cap \partial B_t (x))< \infty, \qquad \mbox{and $\Sigma'$ intersects $\partial B_t (x)$ transversally}.
\end{equation}
Fixing $s<\rho$, we consider the replacement $V''$ of $V'$ in $\an(x, s, t)$, which in this annulus coincides with a smooth $c$-stable surface $\Sigma''$. We remark that $V'$ and $V"$ are integral varifolds by Lemma \ref{rect} and the properties of replacements.

Fix a point $y\in \Sigma'\cap \partial B_t (x)$.

\medskip

\noindent

{\bf Step 1}: {\emph{We claim that there exists a sufficiently small radius $r$, so that 
$$
\Sigma''\cap B_{t}(x)\cap B_{r}(y)=\Sigma'\cap B_{t}(x)\cap B_{r}(y).
$$
}}
Given the local nature of Step 1, by Remark \ref{metricinG} we will assume without loss of generality that the ambient space is $\R^3$.
There are just two cases: either $y\in \mathcal R(\Sigma')$ or  $y\in \mathcal S(\Sigma')$. 

{\em Case 1: Fix $y\in \mathcal R(\Sigma')$.} 
Since $y\in \mathcal R(\Sigma')$, there exists a sufficiently small radius $r$, so that $\Sigma'\cap B_r (y)$ 
is a smooth deformation of a disk and $\gamma=
\Sigma'\cap \partial B_t (x)\cap B_r (y)$ is a $C^{2,\alpha}$ embedded curve.
We recall that, by construction of the replacement, $\Sigma'$ is the smooth limit of a locally $\FF^c$-minimizing sequence $\{\partial \Omega_j'\}_{j\in \N}$ in  $\an(x, \rho, 2\rho)$. This implies that, denoting $\gamma_j=
\partial \Omega_j'\cap \partial B_t (x)\cap B_r (y)$, then \(\gamma_{j}\) are uniformly bounded in $C^{2,\alpha}$ and uniformly embedded, namely
\[
\inf_{j} \inf_{\substack{x\ne y\\ \,x,y\in \gamma_j}}\frac{ \d_{\gamma_j}(x,y)}{|x-y|}>0.
\]
Moreover $\Sigma''$ is obtained as a smooth limit in $\an(x, s, t)$ of a sequence $\{\partial \Omega_j''\}_{j\in \N}$.
Furthermore, since $V''$ is $c$-stationary, with the same argument used to prove \eqref{arg} or \eqref{cond}, we get (up to take a smaller radius $r$)
$$
\partial \Omega_j''\cap \partial B_{t}(x) \cap B_r (y) = \gamma_j,
$$
as $t<2\rho<\lambda/c$, then $H_F^{\partial B_{t}}> 2c$. We apply Lemma \ref{l:min1} to $\{\partial \Omega_j''\}_{j\in \N}$ and deduce that $\partial \Omega_j''$ is smooth up to the boundary $\gamma_j$ in $B_{t}(x) \cap B_r (y)$.  Recalling that $\partial \Omega_j''$ is a c-stable surface in $\an(x, s, t)$, we deduce from Theorem \ref{thm:curvature estimates at the boundary} the following uniform boundary curvature estimates:
\begin{equation*}
 \sup_{\substack{p \in B_{\frac{r}{4}}(y) \cap \Sigma'' \\ \d(p, \gamma_j)< r_1}} r_1 |A^{\Sigma_j''}(p)| \le C.
\end{equation*}
This implies that $\{\partial \Omega_j''\}_{j\in \N}$ converge to $\Sigma''$ smoothly up to the boundary $\gamma$. Standard  regularity theory for Elliptic PDEs implies that 
\begin{equation}\label{continuation}
\Sigma''\cap B_{t}(x)\cap B_{\frac{r}{4}}(y)\supset \Sigma'\cap B_{t}(x)\cap B_{\frac{r}{4}}(y).
\end{equation}
In particular $V''\cap B_{\frac{r}{4}}(y)$ is $c$-stable.

{\em Case 2: Fix $y\in \mathcal S(\Sigma')$.} We observe that, by \eqref{isolato}, there exists a sufficiently small radius $r$, so that $ \mathcal S(\Sigma')\cap B_r (y)\cap \partial B_t (x)=\{y\}$. Up to take a smaller $r$, we can assume $\Sigma'\cap \partial B_t (x)\cap B_r (y)=\gamma= \gamma^+\cup \gamma^-$, where $\gamma^+, \gamma^-$ are $C^{2,\alpha}$ embedded curves, which are tangent each other in $y$ and such that $\partial \gamma^\pm  \cap B_r (y) = \emptyset$. 
Moreover, since $V''\in \cV^c_\infty$, then $\de_{\FF}\Sigma''(X)=\int_{\Sigma''}c\langle X, \nu^{\Sigma''} \rangle \, d\haus^2 +\eta_s(X)$, for every $\mathcal X\in X_c(B_r(y))$, where $\eta_s$ is a finite Radon measure supported on $\gamma$ satisfying $\de_{\FF}\Sigma'(X)=\int_{\Sigma'}c\langle X, \nu^{\Sigma'} \rangle \, d\haus^2 -\eta_s(X)$, for every $\mathcal X\in X_c(B_r(y))$. Furthermore $\Sigma''\cap \partial B_t (x)\cap B_r (y)=\gamma$, otherwise we would have $\Sigma''$ touching $\partial B_{t}(x)$ from the inside. This violates the maximum principle, since $t<2\rho<\lambda/c$, which in turn reads $H_F^{\partial B_{t}}> 2c$. Since $H_F^{\partial B_{t}}> 2c$, we deduce by Proposition \ref{noblow} that, up to take a smaller $r$,
\begin{equation}\label{destra}
 \sup_{\substack{ r'< r}} \frac{ \haus^2(\Sigma''\cap B_{r'}(y))}{\pi (r')^2}<\infty. 
\end{equation}
Moreover, since $y\in \mathcal S(\Sigma')$, up to take a smaller $r$, we get that
\begin{equation}\label{sinistra}
 \sup_{\substack{ r'< r}} \frac{ \haus^2(\Sigma'\cap B_{r'}(y))}{\pi (r')^2}\leq 3. 
\end{equation}
Combining \eqref{destra} and \eqref{sinistra}, we deduce that
$$\sup_{\substack{ r'< r}} \frac{ \|V''\|(B_{r'}(y))}{\pi (r')^2}<\infty, \qquad \mbox{which implies} \qquad \theta^*(V'',y)<\infty. $$
We deduce that the family $\{(\eta_{y,r'})_\# V''\}_{r'<r}$ has uniformly bounded mass and consequently $TV(y, V'')\neq \emptyset$. 
Up to rotation, denoting $p=(p_1,p_2,p_3)\in \B_1$, we can assume that $T_y(\partial B_t(x))=e_1^\perp$, that $\eta_{y,r'}(B_t(x))\subset \{p_1 \geq 0\}$ and that $\dot\gamma/|\dot\gamma|=e_2$. 
Since $y\in \mathcal S(\Sigma')$, by \eqref{isolato} we deduce that for every $\C \in TV(y, V'')$, 
\begin{equation}\label{plane}
\C\res \{p_1 \leq 0\}=2T_y\Sigma'\res \{p_1 \leq 0\}, \qquad \mbox{where $T_y\Sigma' \neq e_1^\perp$}.
\end{equation} 
Fix $\C \in TV(y, V'')$ and denote with $\{r_n\}_n$ the sequence of radii such that $W_n:=(\eta_{y,r_n})_\# V''$ converges to $\C$. 
Moreover, for every $\alpha \in (0,1)$, by \eqref{continuation}, we know that there exists $N(\alpha)\in \N$ such that for every $n\geq N$ it holds $W_n\res \aan(0,\alpha,1/\alpha)\supset \eta_{y,r_n}(\Sigma')\cap \aan(0,\alpha,1/\alpha)$ and $W_n$ is a sequence of $(cr_n)$-stable almost embedded smooth surfaces in $\aan(0,\alpha,1/\alpha)$. Since $r_n \downarrow 0$, by Theorem \ref{T:compact}(ii) $W_n\res \aan(0,\alpha,1/\alpha)$ converge smoothly (with integer multiplicity) to $\Sigma^\alpha$, where $\Sigma^\alpha$ is a stable smooth embedded surface. Since $\C\res \aan(0,\alpha,1/\alpha)=\Sigma^\alpha$, by \eqref{plane} we deduce the following inequality in the sense of varifolds
 $$\C\res \aan(0,\alpha,1/\alpha)\geq 2T_y\Sigma' \res \aan(0,\alpha,1/\alpha), \qquad \forall \alpha\in (0,1).$$
and consequently
$$\C\res \{0\}^c \geq 2T_y\Sigma' \res  \{0\}^c.$$
Since from \eqref{sinistra} it holds $\theta^*(\C,y)<\infty$, we deduce that 
\begin{equation}\label{tangent plane}
\C\geq 2T_y\Sigma'.
\end{equation}
Since both $\C$ and $2T_y\Sigma'$ are stationary, we deduce that $\C':=\C-2T_y\Sigma' \subset \{p_1 \geq 0\}$ is stationary, where again the difference is to be intended in the space of varifolds. 
With the same argument used in the proof of Proposition \ref{noblow} to obtain \eqref{rrrr} and \eqref{rrar}, we can prove that $\C'$ is contained in a wedge $L:=\{|p_3|\leq a p_1, \, p_1 \geq 0\}$ for some $a>0$. We claim that $\C'=0$, and consequently that $\C=2T_{y} \Sigma'$. Indeed if by contradiction $\C'\neq 0$, there exists $\bar h:=\min\{h\geq 0: \{p_1 = h\}\cap \spt(\C')\neq 0\}$. By the maximum principle, we deduce that $\{p_1 = \bar h\}\subset \spt(\C')$. But this cannot be true as $\{p_1 = \bar h\}$ is not entirely contained in the wedge $L$. This is the desired contradiction.

In conclusion we have proved that
$$TV (y, V'')=\{2T_{y} \Sigma'\}.$$

To conclude the proof of this case, we borrow some ideas from the Step 2 of the proof of \cite[Proposition 6.3]{CD}.
Let $g: B_r (y)\to \B_1$ be a diffeomorphism such that 
$$
g(y)=0, \qquad g (\partial B_t (x))\subset \{p_1=0\} \qquad \mbox{and} \qquad
g(\Sigma'')\subset \{p_1>0\}\, ,
$$
where we denoted $p=(p_1,p_2,p_3)\in \B_1$. We will also assume that 
$g (\gamma)=\{(0,p_2, v^+(0,p_2))\}\cup \{(0,p_2, v^-(0,p_2))\}$ and
$g (\Sigma'\cap B_r (y))=\{(p_1, p_2, v^+ (p_1, p_2))\}\cup \{(p_1, p_2, v^- (p_1, p_2))\}$ 
where $v^\pm$ are smooth, tangent in $0$ and $Lv^+=-Lv^-=c$. Moreover, by \eqref{continuation}, for every $\alpha>0$ it holds 
\begin{equation}\label{poi}
g (\Sigma'\cap \an(y, \alpha, r))\cap \{p_1>0\}=g (\Sigma''\cap \an(y, \alpha, r)).
\end{equation}
In the following, with an abuse of notation, we identify $\Sigma, \Sigma',\Sigma'',V, V', V''$ respectively with $g (\Sigma),g (\Sigma'),g (\Sigma''),g_\# V,g_\# V',g_\# V''$.

The equality $TV(y, V'')= \{T_{y} \Sigma'\}$ implies that
\begin{equation}\label{e:L1}
\lim_{p\to 0, p\in \Sigma''} 
\frac{|\langle p, \tau\rangle|}{|p|}=0, \qquad \mbox{where $\tau$ denotes the unit normal to the graph of $v^+$}.
\end{equation}
Indeed assume that \eqref{e:L1} fails; 
then there is a sequence $\{p_n\}\subset \Sigma''$ such that 
$p_n \to 0$ and
$|\langle p_n, \tau\rangle |\geq \alpha |p_n|$ 
for some $\alpha>0$. Set $r_n=|p_n|$. There exists a constant 
$\beta \in (0,1)$ such that $\B_{2\beta r_n} (p_n)\cap T_{y} \Sigma'=\emptyset$. Thus 
$\d (T_{y} \Sigma',\B_{\beta r_n} (p_n))\geq \beta r_n$. By \eqref{poi}, $p_n\in \Sigma'\cap \aan(0,\beta r_n, 1)=V''\res \aan(0,\beta r_n, 1)$.
Since $\Sigma'$ is $c$-stable in $\B_1$, also $V''$  is $c$-stable in  $\aan(0,\beta r_n, 1)$. Hence we apply \cite[page 305, Theorem, Equation (9)]{Allard1983} with $Z=\B_1$ and $M=\Sigma'$: the whole set of assumptions needed is in \cite[page 294, Section 1.4]{Allard1983} and are trivially satisfied. Since $\B_{\beta r_n} (p_n) \subset \aan(0,\beta r_n, 1)$, we have the lower density estimate
$$
\|V''\| (\B_{\beta r_n} (p_n))=\haus^2(\Sigma'\cap \B_{\beta r_n} (p_n))\geq C \beta^2 r_n^2.
$$
This contradicts the fact that $2T_{y} \Sigma'$ is the only element of 
$TV(y, V'')$.

We now consider the unit normal vector field $\nu$ to $\Sigma''$ such that 
$\langle \nu, (0,0,1)\rangle\geq 0$.
We claim that 
\begin{equation}\label{e:L2}
\lim_{p\to 0, p\in \Sigma''} \nu (z) \;=\; \tau \, .
\end{equation}
Indeed let $\sigma$ be the plane $\{(0,p_2, p_3), \, p_2, p_3\in \rn{}\}$, 
assume that $p_n\to 0$ and set $r_n:=\d (p_n,\sigma)$, and 
$\Sigma^n := \eta_{p_n,r_n} (\Sigma''_n \cap \B_{r_n} (p_n))$. Each 
$\Sigma^n$ is a $(cr_n)$-stable surface in $\B_1$ with $cr_n\downarrow 0$, and hence by Theorem \ref{T:compact}(ii), after possibly passing to a 
subsequence,  
$\Sigma^n$ converges smoothly (with integer multiplicity) in 
$\B_{1/2}$ to a stable
surface $\Sigma^\infty$. Since by
\eqref{e:L1} $\Sigma^\infty=T_{y}\Sigma'\cap \B_{1/2}$, then $\nu (p_n)$ converges to $\tau$ as desired.

We deduce that there exists $r>0$ and functions 
$w^+,w^-\in C^1(\{p_1\geq 0\})$ satisfying
$$\Sigma''\cap B_r (y)\;=\;\{(p_1, p_2, w^+(p_1, p_2)),
p_1>0\}\cup \{(p_1, p_2, w^-(p_1, p_2)),
p_1>0\}\, ,
$$
and the boundary conditions $v^+(0, p_2)=w^+(0, p_2)$, $v^-(0, p_2)=w^-(0, p_2)$, $\nabla v^+(0, p_2)=\nabla w^+(0, p_2)$, and $\nabla v^-(0, p_2)=\nabla w^-(0, p_2)$. 
In the variables $p_1, p_2, p_3$, the functions $v^\pm$ and
$w^\pm$ satisfy the same second order uniformly 
elliptic equation. Hence, by standard elliptic PDEs theory we deduce that $v^\pm$ and $w^\pm$ are restrictions of two unique smooth functions $u^\pm$.

\medskip

\noindent
{\bf Step 2}: {\emph{ We claim that $V$ is almost embedded $c$-stable surface in $B_{2\rho}(x)\setminus \{x\}$.}} The strategy to prove the claim is similar to the one used in Step 3 of the proof of \cite[Proposition 6.3]{CD}. However several arguments need to be changed.
Given the local nature of this claim, by Remark \ref{metricinG} we will assume without loss of generality that the ambient space is $\R^3$. 
We first observe that:
\begin{equation}\label{e:claims3}
\mbox{$\Sigma' \cap \overline{S} \cap \partial B_t (x)\neq \emptyset$, \quad for every connected component $S$ of $\Sigma''$.}
\end{equation}
Indeed assume there exists $S$ which does not satisfy \eqref{e:claims3}.  
Since $t<2\rho<\lambda/c$ and consequently $H_F^{\partial B_{t}}> 2c$, the maximum principle implies that 
$\overline{S}\cap \partial B_t (x) \neq \emptyset$. Fix $y$ in
$\overline{S}\cap \partial B_t (x)$. 
As $S$ does not satisfy \eqref{e:claims3}, there exists $s>0$ such that
$$
y\in  \supp (\|V''\|)\cap \partial B_t (x) \qquad \mbox{and}\qquad  
(\supp (\|V''\|) \cap B_s (y)) \subset \overline{B_t (x)}\,.
$$
Since $V''\res B_s (y)$ has bounded anisotropic first variation in $B_s (y)$ and $\|\delta_{\FF}(V''\res B_s (y))\|_{sing}=0$ as a varifold in the domain $B_s (y)$, we can apply \cite[Lemma 4.5]{DKS} with $\Omega=B_s (y)$ to conclude that $\supp (\|V''\res B_s (y)\|)$ is a $(2,c)$ subset of $B_s (y)$. In particular, applying \cite[Proposition 2 (iii)]{DDH} with $\Omega=B_s (y)$ and $N= \overline{B_t (x)} \cap B_s (y)$, we deduce that $H_F^{\partial B_t (x)}(y)\leq c$ (one can verify that the proof of \cite[Proposition 2 (iii)]{DDH} works with the weaker assumption that $N$ has smooth boundary just around the point $y$). This contradicts the fact that $H_F^{\partial B_{t}}> 2c$.

We observe that \eqref{e:claims3} implies that for every $r<\rho$, then $\Sigma'$ can be extended to a surface $\Gamma_r$ in $\an(x, r, 2\rho)$ and for every $r_1<r_2< \rho$, then $\Gamma_{r_1}=\Gamma_{r_2}$ in 
$\an (x, r_2, 2\rho)$.
We can now define our candidate $\Sigma:=\bigcup_r \Gamma_r$, which is a $c$-stable surface 
with $\overline{\Sigma}\setminus \Sigma\subset 
(\partial B_{2\rho} (x)\cup\{x\})$. To conclude Step 2, we need to prove that $V$ coincides with $\Sigma$ in $B_\rho (x)\setminus
\{x\}$. Recall that $V=V'$ in $B_\rho (x)$. Given $y\in (\supp (\|V\|))\cap B_\rho (x) \setminus \{x\}$ we define $r= d (y,x)$.
We claim that if $TV(y, V)=\{\theta \pi\}$, with $\theta>0$, and $\pi$ is a plane transversal 
to $\partial B_r (x)$, then $y\in\Sigma$.
Denoting the second replacement $V''$ of $V'$ in $\an (x, r, t)$, it is enough to show that $y\in \mbox{cl}((\supp (\|V''\|))\setminus \overline{B_r
(x)})$, which in turn implies $y\in \overline{\Sigma}_t\subset \Sigma$.
Assume by contradiction $y\not \in \mbox{cl}((\supp (\|V''\|))\setminus \overline{B_r
(x)})$. In particular there exists $a>0$ such that 
\begin{equation}\label{utile}
((\supp (\|V''\|))\setminus \overline{B_r
(x)})\cap B_a(y)= \emptyset.
\end{equation}
Up to rotation we can assume that $\eta_{y,1/n}(B_r (x)) \overset{n\to \infty}{\longrightarrow} \{p_1> 0\}$. Fix an element $\C \in TV(y,V'')$.
Combining \eqref{utile} with the fact that $V''\res B_r (x) =V\res B_r (x)$, we deduce that $\C=\theta \pi \res \{p_1> 0\}$. Since $\delta_\FF(\C)=0$ and $\theta>0$, we get a contradiction with the Constancy Theorem \cite[Proposition 5]{DDH}.

It is now enough to observe that by Lemma~\ref{trasv}, the set of points $y\in B_\rho (x)$ considered above is dense in $\supp (\|V\|)$.
We deduce that
\begin{equation}\label{hm}
(\supp (\|V\|))\cap B_\rho (x)\setminus \{x\}\;\subset\; \Sigma\, .
\end{equation}
By the last property in Proposition \ref{rep}, we deduce that $\FF (\Sigma\cap B_\rho (x))= \FF(V\res B_\rho (x))$, which by \eqref{hm} and the integrality of $V$ implies that $V=\Sigma$ on $B_\rho (x)\setminus \{x\}$.

\medskip

\noindent
{\bf Step 3}: {\emph{ We deduce the global regularity of V.}}

The function $r(x)$ is characterized in Proposition \ref{p:almost1}. In particular we split the proof of this step in three cases:

{\em Case (a): there exists $\Inj(M)/18>R>0$ such that $r(x)\equiv R$ for every $x\in M$ and there exists $y\in M$ such that $\{\Om^k\}$ is a.m. in $M\setminus \overline{B}_{18R}(y)$.
}\\
By the definition of $\rho$ in \eqref{rhowow}, if we denote
$$s:=\frac 12\min\left \{\lambda/c, R, \frac{1}{\lambda(c+\lambda+4\lambda^3)}\right \},$$
we deduce that  $V$ is smooth, almost embedded $c$-stable surface in $B_{s}(x)\setminus \{x\}$, for every $x \in M$. By compactness of $M$, we can extract from $\{B_s(x)\}_{x\in M}$ a finite cover $\{B_s(x_1), \dots, B_s(x_k)\}$ such that for every $i=1,\dots,k$ there exists $j\neq i$ such that $x_i\in B_s(x_j)\setminus \{x_j\}$. 
This implies that $V$ is smooth, almost embedded surface in the whole ambient space $M$. Using the logaritmic cutoff, one can easily verify that $\{x\}$ has zero $2$-capacity in a $2$-dimensional smooth surface. Then, by a simple capacity argument, $V$ is a $c$-stable surface in $B_{s}(x)$, for every $x \in M$. 
Moreover there exists $y\in M$ such that $\{\Om^k\}$ is a.m. in $M\setminus \overline{B}_{18s}(y)$ and since $\partial \Om^k$ converges to $V\in \cV_\infty^c$ and $V$ is smooth everywhere, then $V$ is a $c$-stable surface in $M\setminus \overline{B}_{18s}(y)$. This can be shown using the same argument as at the end of Proof of Lemma \ref{l:min}. We repeat the argument for the sake of exposition:

Assume by contradiction that there exists a smooth volume preserving vector field $X$ compactly supported in $M\setminus \overline{B}_{18s}(y)$ such that $\delta^2_{\FF} {V}(X)=-C<0$. 
There exists a map $\varphi:t\in \R \mapsto \varphi_t\in C^\infty (M\setminus \overline{B}_{18s}(y),M\setminus \overline{B}_{18s}(y))$ solving the following ODE:
\begin{equation*}
\begin{cases}\frac{\partial \varphi_t(x)}{\partial t}=X(\varphi_t(x))&\quad \forall x\in M,\\
\varphi_0(x)=x &\quad \forall x\in M.
\end{cases}
\end{equation*}
We set
$$\Omega^{k}_t:=(\varphi_t)(\Omega^{k}).$$
Since $X$ is a volume preserving vector field, $\varphi_t$ has $C^3$-norm bounded uniformly in $t\in [0,1]$, $\delta_{\FF} {V}(X)=0$, $\delta^2_{\FF} {V}(X)=-C<0$ and $\partial \Omega^k \rightharpoonup V$ as varifolds, we get that there exists $\e>0$ and $k_0\in \N$ such that
\begin{equation*}
\Fc({\Omega}^k_t)\leq \Fc({\Omega}^k) - \frac C2t^2 \qquad  \forall t<\e, \quad \forall k \geq k_0.
\end{equation*}
We can easily define an homotopy $\psi(t,x)=\varphi_{\e t}(x)$ for every $(t,x)\in [0,1]\times (M\setminus \overline{B}_{18s}(y))$ and $\psi(t,x)=x$ for every $(t,x)\in [0,1]\times  \overline{B}_{18s}(y)$. Then, denoting $\tilde{\Omega}_t^k:=\psi(1,{\Omega}^k)$, it is easy to see that $\{\tilde{\Omega}_t^k\}_{t\in [0,1]}$ satisfies all properties \eqref{m1},\eqref{m2},\eqref{m3},\eqref{m4} for $k$ large enough. This contradicts the almost minimality of the sequence $\{\Omega^{k}\}$ in $M\setminus \overline{B}_{18s}(y)$.

We observe that the radius $18s$ should be multiplied by a factor, by the definition of $s$. However the factor is a geometric constant depending just on $F$ and $M$, hence we can omit it without loss of generality.

{\em Case (b): $r(x)\equiv \Inj(M)/18$ for every $x\in M$.} We denote
$$s:=\frac 12\min\left \{\lambda/c,  \Inj(M)/18, \frac{1}{\lambda(c+\lambda+4\lambda^3)}\right \}.$$
With the same argument of Case (a), we can show that $V$ is a smooth, almost embedded surface in the whole ambient space $M$ and that $V$ is a $c$-stable surface in $B_{s}(x)$, for every $x \in M$.

{\em Case (c):  there exists $p\in M$ such that $r(x)=d(x,p)$ for $x\neq p$ and $r(p)=\infty$.}

Again, by the definition of $\rho$ in \eqref{rhowow}, if we denote
$$s(x):=\begin{cases}
                                   \frac 12\min\{\Inj(M),\lambda/c, d(x,p), \frac{1}{\lambda(c+\lambda+4\lambda^3)}\} & \text{if $x\neq p$} \\
                                   \frac 12\min\{\Inj(M),\lambda/c, \frac{1}{\lambda(c+\lambda+4\lambda^3)}\} & \text{if $x=p$} 
                                     \end{cases},$$
we deduce that  $V$ is smooth, almost embedded $c$-stable surface in $B_{s(x)}(x)\setminus \{x\}$, for every $x \in M$.
By compactness of $M$, we can extract from $\{B_s(x)\}_{x\in M\setminus B_{s(p)/2}(p)}\cup \{B_{s(p)}(p)\}$ a finite cover $\{B_{s_1}(x_1), \dots, B_{s_k}(x_k), B_{s(p)}(p)\}$ such that  for every $i=1,\dots,k$ it holds $x_i\neq p$ and there exists $j\neq i$ such that $x_i\in B_{s_j}(x_j)\setminus \{x_j\}$.  
This implies that $V$ is smooth, almost embedded $c$-stable surface in $M\setminus \{p\}$. Notice that the stability is not a consequence of the local stbaility (the second variation is not linear), but of the fact that $r(p)=\infty$.
\end{proof}

\section{Technical Propositions}\label{br}
In this section, we prove some techincal propositions of independet interest. Since they are valid in every dimension, i.e. for hypersurfaces in $(n+1)$-dimensional manifolds, just in this section we will denote with $n$ the dimension of the hypersurface and $n+1$ will be the dimension of the ambient manifold. We remark that the notation introduced in Section \ref{s:prel} easily extends to every dimension.
 \begin{Lemma}\label{stimatouchingset}
Let $M$ be an $(n+1)$ dimensional  $C^2$ manifold. Let $\Sigma \subset M$ be a $C^2$ almost embedded hypersurface in $M$ with $H_F^\Sigma\equiv c$. Then the touching set $\mathcal S(\Sigma )$ is $\haus^{n-1}$-locally finite, i.e., for every compact set $K\subset M$ there exists $C>0$ such that
$$\haus^{n-1}(\mathcal S(\Sigma)\cap K) \leq C.$$
\end{Lemma}
\begin{proof}
Let $x \in \mathcal S(\Sigma)$. There exists $r(x)>0$ such that $\Sigma \cap B_{r(x)}(x)$ decomposes as two distinct graphs $u_1\leq u_2$ over the common tangent plane $T_x\Sigma$ (compare with the proof of Theorem \ref{T:compact}). We define $u:=u_1-u_2$ and, since $Lu_1=-Lu_2=c$, then $Lu=2c$. Moreover, by standard regularity for elliptic PDEs, up to shrink $r(x)$, we have that $u, u_1, u_2\in C^2$. We further observe that $\mathcal S(\Sigma)\cap B_{r(x)}(x) \subset \{u=0, Du=0\}$.
Since at every $p \in \{u=0, Du=0\}\cap B_{r(x)}(x)$ the equation $Lu(p)=2c>0$ has just the second order terms, we deduce that $D^2u(x)\neq 0$. In particular, there exists $v\in T_x\Sigma$ such that $\partial_v(Du)(x) \neq 0$. By the implicit function theorem, we deduce that there exists $r'(x)<r(x)$, an $(n-1)$-plane $V$ and $f\in C^1(V, V^\perp \cap T_x\Sigma)$ such that, denoting with $\mathbf{G}_{f}$ the graph of $f$ over $V$, it holds:
$$\{Du=0\}\cap  B_{r'(x)}(x) = \mathbf{G}_{f}\cap B_{r'(x)}(x).$$
Since $f\in C^1$, there exists $C_x>0$ such that
\begin{equation*}
\begin{split}
\haus^{n-1}(\mathcal S(\Sigma)\cap B_{r'(x)}(x)) &\leq \haus^{n-1}(\{u=0, Du=0\}\cap  B_{r'(x)}(x))\\
&\leq  \haus^{n-1}(\{Du=0\}\cap  B_{r'(x)}(x)) \\
&= \haus^{n-1}(\mathbf{G}_{f}\cap B_{r'(x)}(x))\leq C_x(r'(x))^{n-1}.
\end{split}
\end{equation*}
Since $\mathcal S(\Sigma)$ is a closed set in $\Sigma$, then for every compact set $K\subset M$ the set $K\cap \mathcal S(\Sigma)$ is compact. We can extract a finite covering of $K\cap \mathcal S(\Sigma)$ from the family $\{B_{r'(x)}(x)\}_{x\in K\cap \mathcal S(\Sigma)}$ of balls $B_1, \dots B_k$ of radii $r_1,\dots,r_k$ centered in points $x_1,\dots x_k$. Setting $C:=\max\{C_{x_1},\dots, C_{x_k}\}$, we deduce that 
$$\haus^{n-1}(K\cap \mathcal S(\Sigma))\leq \sum_{i=1}^k \haus^{n-1}(\mathcal S(\Sigma)\cap B_i)\leq \sum_{i=1}^k C r_i^{n-1}=:C.$$
This concludes the proof.
\end{proof}

\begin{Remark}
Given the local nature of Proposition \ref{noblow}, Proposition \ref{noblow}, Theorem \ref{thm:curvature estimates at the boundary}, Lemma \ref{trasv}, Lemma \ref{nonconc} and Lemma \ref{lem:blowup2}, by Remark \ref{metricinG} we can replace $\R^{n+1}$ with a $C^3$ $(n+1)$-dimensional manifold $M$.
\end{Remark}

\begin{Proposition}\label{noblow}
Let $\Omega \subset \R^{n+1}$ such that $\partial \Omega \cap \overline{B_{2R}}$ is $C^{3}$ and satisfies $H_F^{\partial \Omega}> 2c$ in \(B_{2R}\).  Let $\Gamma$ be a $C^{2,\alpha}$ embedded $(n-1)$-submanifolds in $\partial \Omega \cap B_{2R}$ with $\partial \Gamma^\pm  \cap B_{2R} = \emptyset$. Let $\Si$ be a $C^2$ almost embedded manifold in $\Omega$ such that 
\begin{equation}\label{bb01}
\de_{\FF}\Sigma(X)=\int_{\Sigma}c\langle X, \nu^\Sigma \rangle \, d\haus^2 +\eta_s(X), \qquad \forall \mathcal X\in X_c(B_R),
\end{equation}
where $\eta_s$ is a finite Radon measure supported on $\Gamma$.
Then there exists a constant $C$ and a radius $r_0>0$ depending only on $F, \Omega, x, \Gamma$  such that
$$ \sup_{\substack{ r< r_0}} \frac{ \haus^n(\Si\cap B_r(x))}{w_n r^n} \le C < \infty. 
$$
\end{Proposition}
\begin{proof}
The proof is identical to the proof of \cite[Proposition 6]{DDH}. Indeed the statement of Proposition \ref{noblow} differs from the one of  \cite[Proposition 6]{DDH} just in the assumption \eqref{bb01}, which is different from the corresponding one in \cite{DDH}, i.e. $\de_{\FF}\Sigma=0$ and $\partial \Si \cap B_R = \Gamma$. However one can easily verify by the maximum principle that \eqref{bb01} is enough for the proof.
\end{proof}

\begin{Proposition}\label{noblow2}
Let $\Omega \subset \R^{n+1}$ such that $\partial \Omega \cap \overline{B_{2R}}$ is $C^{3}$ and satisfies $H_F^{\partial \Omega}> 2c$ in \(B_{2R}\).  Let $\Gamma= \Gamma^+\cup \Gamma^-$, where $\Gamma^+, \Gamma^-$ are $C^{2,\alpha}$ embedded $(n-1)$-submanifolds in $\partial \Omega \cap B_{2R}$, which are tangent each other in $x\in \Gamma^+\cap \Gamma^-$ and such that $\partial \Gamma^\pm  \cap B_{2R} = \emptyset$. Furthermore, let $\Si$ be a $C^2$ almost embedded manifold in $\Omega$ such that 
\begin{equation}\label{bb0}
\de_{\FF}\Sigma(X)=\int_{\Sigma}c\langle X, \nu^\Sigma \rangle \, d\haus^2 +\eta_s(X), \qquad \forall \mathcal X\in X_c(B_R),
\end{equation}
where $\eta_s$ is a finite Radon measure supported on $\Gamma$.
Then there exists a constant $C$ and a radius $r_0>0$ depending only on $F, \Omega, x, \Gamma$  such that
\begin{equation}\label{eq:blowup} \sup_{\substack{ r< r_0}} \frac{ \haus^n(\Si\cap B_r(x))}{w_n r^n} \le C < \infty. \end{equation}
\end{Proposition}

\begin{proof}
The proof is analogous to the proof of \cite[Proposition 6]{DDH}. Indeed the statement of Proposition \ref{noblow2} differs from the one of  \cite[Proposition 6]{DDH} in two assumptions. The first difference is \eqref{bb01}, which is weaker than the corresponding assumption in \cite{DDH} $\partial \Si \cap B_R = \Gamma$. However one can verify that \eqref{bb01} is enough for the proof. The second difference is that here $\Gamma$ is the union of two touching curves $\Gamma^+, \Gamma^-$, while in \cite{DDH} $\Gamma$ is a single curve. However, one can slightly modify the proof of \cite[Proposition 6]{DDH} as follows. 

Assume the conclusion \ref{eq:blowup} does not hold, then there exists a sequence $\Si_k, r_k$ satisfying
$$
0<r_k< \frac{1}{k} \qquad \mbox{and} \qquad \frac{\haus^n(\Si_k\cap B_{r_k}(0))}{r_k^n} > k. 
$$
Let us denote $\gamma^+, \gamma^-$ the projection of respectively $\Gamma^+, \Gamma^-$ onto $\{ x_{n+1} = 0 \}$. Up to subsequences, and performing if necessary a rotation of $B_2$, the normals $\hat{\nu}^+(0)$ and $\hat{\nu}^-(0)$ of respectively $\gamma^+$ and $\gamma^-$ in the plane $\{x_{n+1} =0 \}$ at $0$ satisfy
\begin{equation}\label{normal}
\hat{\nu}^+(0) = e_{n}, \qquad \mbox{and} \qquad \hat{\nu}^-(0) = - e_{n}
\end{equation}
We can choose  $r_0>0$ such that for all $k \in \N$ we have 
$$r_0 < \min\left\{\frac{1}{\norm{A^{\gamma^+} }_{\infty}},\frac{1}{\norm{A^{\gamma^-} }_{\infty}},\frac12\right\},$$
 where $A^{\gamma^\pm}$ is the second fundamental form of $\gamma^\pm$. 
Defining
\[ B^{\pm} := B^n_{r_0}( r_0\hat{\nu}^\pm(0)) \subset \{x_{n+1} =0 \}, \]
our choice of $r_0$ implies that $\overline{B^\pm} \cap  (\gamma^+\cup \gamma^-) = \{0\}$.
 For each $0\le s \le \delta$, using \cite[Assumption 1]{DDH} and the notation therein (in particular the definition of $\Phi$), we can consider the unique solutions $u^\pm_{s} \in C^{2,\alpha}(B^\pm)$ of the following problems 
\[ \left\{\begin{aligned}
L u^\pm_{s} &= s &&\text{ in } B^\pm \\
u_{s}^\pm &= \Phi &&\text{ on } \partial B^\pm \, .
\end{aligned}\right.\]

From the Hopf-maximum principle, it follows that if $s> h_{max}$ then $u^\pm_{s} < \Phi$ and if $s< h_{min}$ then $u^\pm_{s} > \Phi$. We observe that the graphs of $u^\pm_{s}$ never touch $\Si_k$ in the interior of $B^\pm \times \R$ when $s \ge \frac12 h_{min}$. Indeed, if $s> h_{max}$ this easily follows by $\Si_k \subset \Omega$. Hence we can assume without loss of generality that there exists a first (the biggest) $\frac12 h_{min} < s\le h_{max}$ where the graph $u^+_{s}$ touches $\Si_k$ at a point $q=(y,u^+_{s}(y))$. Then $\Si_k$ is the graph of a map $f_k$ over the plane $\{x_{n+1} = 0 \}$ in a neighborhood of $y$. Since $\Si_k$ is $c$-stationary, we have $L(f_k)=c<\frac 12 h_{min}<s$, which contradicts the maximum principle, as $u^+_s\leq f_k$. One can argue analogously for $u^-_{s}$.

We now set $
u^\pm:=u^\pm_{\frac12 h_{min}}$.
 Hopf boundary point lemma allows us to compare $\Phi$ with $u^\pm$ at $0$. We deduce that there exists $c_H>0$ depending only on  $F$ and $\partial \Omega$ such that  
\begin{equation}\label{rrrr} \min\left\{ \frac{\partial u^+(0)}{\partial \hat{\nu}^+(0)} - \frac{\partial \Phi(0)}{\partial \hat{\nu}^+(0)}, \frac{\partial u^-(0)}{\partial \hat{\nu}^-(0)} - \frac{\partial \Phi(0)}{\partial \hat{\nu}^-(0)} \right\} > c_H. \end{equation}
Now we consider the blow-up sequences
$$
\Si_k':= \frac{1}{r_k}\Si_k, \qquad  \Omega_k:= \frac{1}{r_k}\Omega, \qquad \Gamma_k^\pm= \frac{1}{r_k}\Gamma^\pm, \qquad \Gamma_k=\Gamma_k^+\cup \Gamma_k^-, \qquad  \gamma^\pm_k = \frac{1}{r_k}\gamma^\pm,$$
$$
 d_k^\pm:x \in \frac{1}{r_k}B^\pm \to \frac{u^\pm(r_k x) - \Phi(r_k x)}{r_k} .
$$
Using \eqref{normal}, \eqref{rrrr}, since $\frac{1}{r_k}B^\pm \to \R^n \cap \{ \pm x_n \ge 0 \}$, and $d^\pm_k=0$ on $\partial (\frac{1}{r_k}B^\pm)$, we conclude that
$$\mbox{$\Omega_k \to \{ x_{n+1} \ge 0\}$; \qquad  $\gamma_k^\pm \to \{ x_n =0 \}$;}$$
\begin{equation}\label{rrar}
\mbox{$d^\pm_k(x) \to a^\pm x_n$ for every $x \in \R^n \cap \{ \pm x_n \ge 0 \}$ for some $a^+, -a^->c_H$.}
\end{equation}
One can now conclude the proof analogously as in \cite[Proposition 6]{DDH}.

\end{proof}

\begin{Theorem}\label{thm:curvature estimates at the boundary} Let $\Omega \subset \R^{3}$ s.t. $\partial \Omega \cap \overline{B_{2R}}$ is $C^{3}$ and satisfies $H_F^{\partial \Omega}> 2c$ in \(B_{2R}\).  Let $\Gamma$ be a $C^{2,\alpha}$ embedded curve in $\partial \Omega \cap B_{2R}$  with $\partial \Gamma \cap  B_{2R} = \emptyset$. Furthermore let $\Si$ be a $c$-stable, $C^2$ regular surface in $\Omega$ such that $\partial \Si \cap B_R = \Gamma$. Then there exists a constant $C>0$ and a radius $r_1>0$ depending only on $F, \Omega, \Gamma$  such that 
\begin{equation*}
 \sup_{\substack{p \in B_{\frac{R}{2}} \cap \Omega \\ \d(p, \Gamma)< r_1}} r_1 |A(p)| \le C.
\end{equation*}
Moreover the constants \(C\) and \(r_1\)  are uniform as long as \(\Omega\), \(\Gamma \) and \(F\) vary in compact classes.
\end{Theorem}
\begin{proof} 
The proof can be repeated verbatim as the proof of \cite[Theorem 4.1]{DDH}. We highlight here only the differences. One can easily verify by the maximum principle (as in the proof of Proposition \ref{noblow} and Proposition \ref{noblow2}) that \cite[Proposition 6]{DDH} still holds under the assumption that $H_F^{\partial \Omega}> 2c$ and $\Si$ is a $c$-stationary surface (rather than stationary). Moreover, by Remark \ref{stablee}, the $c$-stable surface $\Si$ satisfies \cite[Equation (41)]{DDH}.
We conclude observing that the proof \cite[Lemma 4.3]{DDH} applies also to $c$-stable surfaces, provided \cite[Equation (41)]{DDH} holds.
\end{proof}

\begin{Lemma}\label{trasv}
Consider a point $x\in M$, an integral $n$-varifold $V\in \cV_\infty^c$ in $\R^{n+1}$, and the set
\begin{equation*}
\begin{split}
Q\;=\;\{y \in \supp(\|V\|):& \mbox{ $TV(y, V)=\{\theta \pi\}$}, \\
& \mbox{with $\theta>0$, and $\pi$ is a plane transversal 
to $\partial B_{\d(x,y)} (x)$}\}\,.
\end{split}
\end{equation*}
If $\rho<\frac{1}{\lambda(c+\lambda+\lambda^3(n+2))}$, then $Q$ is dense 
in $(\supp (\|V\|))\cap B_\rho (x)$.
\end{Lemma}
\begin{proof} Being $V$ integral, we have $V=\theta \H^n \res K \otimes \delta_{T_xK}$ for some $n$-rectifiable set $K$ and some 
Borel function $\theta:K\to \N$. Assume by contradiction that the lemma is false. In particular, there exists $y\in B_\rho (x)\cap \supp (\|V\|)$ and $t>0$
such that $T_zK$ is tangent to $\partial B_{\d(z,x)} (x)$, for any $z\in B_t (y)$.
We choose $t$ so that $B_t (y)\subset B_\rho (x)$. Let $f$ be a smooth nonnegative function in $C^\infty_c (B_t (y))$ with $f=1$ on $B_{t/2} (y)$.
We define the vector field
$$X(z):=f(z)D_{\nu}G(z,\tilde{\nu}(z)), \quad \mbox{where $\tilde{\nu}(z):=\nu^{\partial B_{\d(x,z)}(x)}(z)$}.$$
We observe that, by \eqref{costper},
\begin{equation}\label{stim}
|X(z)|\geq \frac{f(z)}{\lambda}.
\end{equation}
We compute the following contradiction
\begin{equation*}
\begin{split}
&\delta_\FF V(X)=\int \langle D_z G(z,\nu), X(z)\rangle + (G(z,\nu)Id - D_\nu G(z,\nu)\otimes \nu):DX(z)dV(z,\nu)\\
&\geq -\lambda \int |X|d\|V\|+\int_{K\cap B_t (y)}  \theta(z)(G(z,\tilde{\nu}(z))Id - D_\nu G(z,\tilde{\nu}(z))\otimes \tilde{\nu}(z)):\\
&:(Df(z)\otimes D_{\nu}G(z,\tilde{\nu}(z))+f(z)D^2_{\nu}G(z,\tilde{\nu}(z))D\tilde{\nu}(z)+f(z)D^2_{z\nu}G(z,\tilde{\nu}(z)))d\H^n(z)\\
&\geq -\lambda \int |X|d\|V\|+\int_{K\cap B_t (y)}  \theta(z)(G(z,\tilde{\nu}(z))\langle Df(z),  D_{\nu}G(z,\tilde{\nu}(z))\rangle +\\
& -G(z,\tilde{\nu}(z))\langle Df(z),  D_{\nu}G(z,\tilde{\nu}(z))\rangle +f(z)G(z,\tilde{\nu}(z))H_F^{\partial B_{\d(x,z)}(x)}(z)-f(z)\lambda^2(n+2))d\H^n(z)\\
&\geq -(\lambda+\lambda^3(n+2)) \int |X|d\|V\|+\int_{K\cap B_t (y)}  \theta(z)f(z)G(z,\tilde{\nu}(z))\frac{1}{\lambda \d(x,z)}d\H^n(z)\\
&\geq (\frac{1}{\lambda \rho}-(\lambda+\lambda^3(n+2))) \int |X|d\|V\|> c\int |X|d\|V\|.
\end{split}
\end{equation*}
where in the first inequality we used \eqref{costper}, in the second inequality we used \eqref{eulercod1} and $D\tilde{\nu}(z)[\tilde{\nu}(z)]=0$, \eqref{cost per area} and \eqref{costper}, in the third inequality we used \eqref{stim}, Remark \ref{palla} and the assumption $\rho<\frac{1}{\lambda(c+\lambda+\lambda^3(n+2))}$.
\end{proof}

\begin{Lemma}\label{nonconc}
There exists $\alpha\in (0,1)$ and  $r_0>0$ (depending just on $F$) such that, for every $n$-varifold $V\in \cV_\infty^c$ and for every $x_0\in \R^{n+1}$
$$\|V\|(B_r(x_0))\leq \alpha \|V\|(B_{2r}(x_0)), \qquad \forall r<r_0.$$
 \end{Lemma}
\begin{proof}
By Remark \ref{duality}, we recall that
 $$\delta_\FF V(X)=\int \langle D_x G(x,\nu), X(x)\rangle + (G(x,\nu)Id - D_\nu G(x,\nu)\otimes \nu):DX(x)dV(x,\nu).$$
Up to translation, we can assume $x_0=0$.
Since $V\in \cV_\infty^c$, we have
\begin{equation}\label{one}
\left | \int \langle D_x G(x,\nu), X(x)\rangle + (G(x,\nu)Id - D_\nu G(x,\nu)\otimes \nu):DX(x)dV(x,\nu) \right|\leq c \int  |X|\,d\|V\|.
\end{equation}
We consider $X(x):=\phi(|x|)x$, where $\phi :\R\to [0,\infty)$ satisfies $\phi(s)=1$ if $s\leq r$, $\phi(s)=0$ if $s\geq 2r$ and $\phi'(s)\leq 2/r$ if $s\in (r,2r)$. It follows that $DX(x)=\phi(|x|)Id+\phi'(|x|)\frac{x\otimes x}{|x|}$. Then \eqref{one} reads
\begin{equation*}\label{two}
\begin{split}
&\left |\int  nG\phi(|x|)-\langle D_\nu G,\nu\rangle \phi(|x|) - \langle \nu, x/|x|\rangle \langle D_\nu G,x\rangle \phi'(|x|)+G|x| \phi'(|x|)  dV(x,\nu)\right|\\
&\qquad \leq (c+\lambda) \int  \phi(|x|)|x|\,d\|V\|,
\end{split}
\end{equation*}
which, using the one-homogeneity of $G$ reads
\begin{equation*}\label{three}
\begin{split}
&\frac{n-1}{\lambda} \int  \phi(|x|)dV(x,\nu)\\
&\leq  \int  \left | \langle \nu, \frac{x}{|x|}\rangle \langle D_\nu G(x,\nu),\frac{x}{|x|}\rangle +G(x,\nu)\right |\phi'(|x|)|x|dV(x,\nu)+ 2r(c+\lambda) \|V\|(B_{2r}).
\end{split}
\end{equation*}
By \eqref{eq:sconv}, we deduce that
\begin{equation*}\label{four}
\frac{n-1}{\lambda} \|V\|(B_{r})\leq  \int  (G(x,\nu) +G(x,\frac{x}{|x|}))\phi'(|x|)|x|dV(x,\nu)+ 2r(c+\lambda) \|V\|(B_{2r}),
\end{equation*}
or equivalently
\begin{equation*}\label{five}
 \|V\|(B_{r})\leq  \frac{2\lambda^2}{n-1}\|V\|(\an(r,2r))+ 2r(c+\lambda) \|V\|(B_{2r}).
\end{equation*}
Denoting $K:=\frac{2\lambda^2}{n-1}$ and adding on both side $K \|V\|(B_{r})$, we deduce that
$$ 
\|V\|(B_{r})\leq  \frac{K+ 2r(c+\lambda)}{K+1} \|V\|(B_{2r}).
$$
Choosing $r<r_0:=\frac 1{4 (c+\lambda)}$ and $\alpha:=\frac{K+ \frac 12}{K+1}$, we conclude the claim.
\end{proof}

\section{Appendix: A boundary regularity theorem}\label{boundreg}

\begin{Theorem}\label{boundregt}
Fix $y\in \partial_+ An$ and assume that the finite perimeter set $\Omega$ minimizes $\FF^c$ in the class $\mathcal{P}(\Omega, B_{\rho} (y)\cap An)$ of the finite perimeter sets equal to $\Omega$ outside of $B_{\rho} (y)\cap An$. Moreover assume that $\partial \Omega\cap  B_{\rho} (y)\setminus An$ is smooth embedded and $\partial \Omega\cap  B_{\rho} (y)\cap  \partial_+ An$ is a smooth embedded curve $\gamma$. Then there exists a smaller ball $B_{\sigma} (y)\subset B_{\rho} (y)$ such that $\partial \Omega\cap B_\sigma (y)\cap An$ is smooth up to the boundary $\partial \Omega\cap \partial_+ An \cap B_\sigma (y)$. 
\end{Theorem}
\begin{proof}
Thanks to \cite[Theorem 0.1]{DS}, it is enough to prove that the lower density $\theta_*(\partial \Omega \cap An, y)\leq \frac 12$. To this aim, it is enough to show the existence of a blowup  $\C \in TV(y,\partial \Omega \cap An)$ which is half a plane.
This is the content of Lemma \ref{lem:blowup2}.  
\end{proof}

In the following, given $\Omega \in \C(\R^3)$ and $x\in \R^3$, we denote with \(TV(x,\Om)\) the sets of all the subsequential limits, as  $r \to 0$, of $\eta_{x,r}(\Om) \in \C(\R^3)$.

To prove the following Lemma \ref{lem:blowup2}, we closely follow the strategy of \cite[Lemma 5.4]{DM}, which is the analogous of \cite[Lemma 4.5]{hardt}.

\begin{Lemma}\label{lem:blowup2}
Under the assumptions of Theorem \ref{boundregt}, denoting $H:=\{x_1>0\}$, there exists  \(\nu\in \mathbb S^{2}:=\{|w|=1\}\) such that up to rotations
\begin{equation}\label{c1}
H\cap \{\langle \nu , x\rangle \le 0 \}\in TV(y,\Omega \cap An)\,.
\end{equation}
\end{Lemma}

\begin{proof} We denote  $G^+=G\cap H$ for every $G\subset\R^3$. Up to a translation and a rotation, we assume $y=0$ and  \(e_{2}=\dot \gamma/|\dot \gamma|\), hence arguing as in the proof of Proposition \ref{noblow}, compare also with \cite[Lemma 5.3]{DM}, we deduce the existence of $L=L(F,c)$ such that
\begin{equation}\label{c2}
\sup_{(\partial E)^+}\,\frac{ |\langle x, e_3\rangle|}{\langle x , e_1\rangle}\le L\,,\qquad\forall E\in TV(y,\Omega)\,.
\end{equation}
Moreover
\begin{equation}
  \label{c5}
 \pa E \cap \pa H = \{0\}\times \R \times\{0\}\,,\qquad\forall E\in TV(y,\Omega)\,.
\end{equation}
We define $\xi:TV(y,\Omega)\to[-L,L]$ as
\begin{eqnarray*}\label{c6}
\xi(E)=\inf_{(\pa E)^+}\,\frac {\langle x, e_3\rangle}{\langle x , e_1\rangle}\,,&&\qquad \forall E\in TV(y,\Omega)\,.
\end{eqnarray*}
One can easily check that $\xi$ is upper semicontinuous on $TV(y,\Omega)$ with respect to the $L^1_{{\rm loc}}(\R^3)$ convergence. Indeed, if \(E_{h},E\in TV(y,\Omega)\) and $E_h\to E$ in $L^1_{\rm loc}(H)$, then, for every \(x\in H\cap \pa E\) there exist \(x_h\in H\cap\pa E_h\), $h\in\N$, such that \(x_h \to x\) as $h\to\infty$. Hence,
\[
\frac{\langle x, e_3\rangle}{\langle x, e_1 \rangle}=\lim_{h\to\infty} \frac{\langle x_h, e_3\rangle}{\langle x_h , e_1\rangle}\ge \limsup_{h\to\infty} \xi (E_h)\,,
\]
as claimed. Since $TV(y,\Omega)$ is compact in $L^1_{{\rm loc}}(\R^3)$, we deduce the existence of $E_1\in TV(y,\Omega)$ such that
\begin{equation}\label{c7}
\xi(E_1)\ge\xi(E)\,,\qquad\forall E\in TV(y,\Omega)\,.
\end{equation}

Let us fix \(\alpha \in (-\pi/2,\pi /2)\) so that $\tan \alpha=\xi(F_1)$ and set
\[
\nu_1=\cos \alpha\, e_3-\sin \alpha \,e_1\in \mathbb S^{2}\,,\qquad H_1=\Big\{x\in H:\langle x, \nu_1\rangle \le0\Big\}\,.
\]
We now prove that
\begin{eqnarray}\label{ccccc}
(\pa H_1)^+ \subset (\pa E_1)^+\,.
\end{eqnarray}
Indeed, by definition of $\xi$, it holds
\begin{equation}
  \label{c8}
  (\pa E_1)^+\subset  H_1\,.
\end{equation}
Moreover, denoting with $w:\{z\in\R^{2}:p_1>0\}\to [-\infty,+\infty)$ the function satisfying
\[
w(z)=\inf\Big\{t\in\R:  (z,t)\in \pa E_1\Big\}\,,\qquad \forall z\in\{z\in\R^{2}:p_1>0\}\,,
\] 
we deduce from  \eqref{c5}, \eqref{c8}, and the lower semicontinuity of $w$ that
\begin{eqnarray}
  \label{c9}
  \Big\{x\in H:x_3\le w(x_1,x_2)\Big\}\subset E_1\,,&&
  \\
  \label{c10}
  w(x_1,x_2)\ge \xi(E_1)\,x_1\,,&&\qquad\forall (x_1,x_2)\in\{z\in\R^{2}:p_1>0\}\,.
\end{eqnarray}
If \eqref{ccccc} fails, then there exists $\bar x\in (\pa E_1)^+$ such that
\begin{equation}
  \label{c11}
  w(\bar x_1,\bar x_2)> \xi(E_1)\,\bar x_1\,.
\end{equation}
By \eqref{c10} and \eqref{c11}, if we set \(\bar r=|(\bar x_1,\bar x_2)|\), $\bar z=(\bar x_1,\bar x_2)$ and $D_{\bar r}=B_{\bar r}\cap \R^2\times \{0\}$, then we can find \(\vphi\in C^{1,1}(\pa(D_{\bar r}^+))\) such that
\begin{eqnarray}\label{c12}
w(z)\ge \vphi (z)\ge \xi(E_1)\,\langle z, e_1\rangle\,,&&\qquad\forall z\in \pa (D_{\bar r}\cap H)
\\
\vphi(\bar z)>\xi(E_1)\, \langle \bar z, e_1\rangle\,.
\end{eqnarray}
In particular, $\vphi=0 $ on $D_{\bar r}\cap\pa H$. By part two of \cite[Lemma 2.11]{DM}, there exists \(u\in C^{1,1}(D_{\bar r}^+)\cap\Lip(\mbox{cl}(D_{\bar r}^+))\) such that, if we set $G_0^\#(p)=G_0(p,-1)$ for $p\in\R^{2}$, then
\begin{equation*}
\begin{cases}
\Div (\nabla_{p} G_0^\#(\nabla u))=0\,,\qquad&\textrm{in \(D_{\bar r}^+\)}\,,
\\
u=\varphi\,,&\textrm{on \(\pa (D_{\bar r}^+)\)}\,,
\end{cases}
\end{equation*}
with
\begin{equation}\label{c14}
|\nabla u(0)|=|\langle \nabla u(0), e_1\rangle |>\xi(E_1)\,.
\end{equation}
Since $\Omega$ minimizes $\FF^c$ in the class $\mathcal{P}(\Omega, B_{\rho} (y)\cap An)$, then the blowup $E_1$ satisfies
\begin{equation}\label{c3}
\FF(E_1,\mbox{cl}(H))\leq \FF(F,\mbox{cl}(H)), \quad \mbox{whenever } E_1 \Delta F \subset H,
\end{equation}
and in particular:
$$\FF(E_1,\mbox{cl}(D_{\bar r}^+\times\R))\leq \FF(F,\mbox{cl}(D_{\bar r}^+\times\R)), \quad \mbox{whenever } E_1 \Delta F \subset D_{\bar r}^+\times\R.$$ 
This, combined with \eqref{c9} and \eqref{c12}, allows us to apply \cite[Lemma 2.12]{DM} to infer that
\begin{equation}
  \label{c15}
 \Big\{(z,t)\in D_{\bar r}^+\times\R:t\le u(z)\Big\} \subset E_1\cap(D_{\bar r}^+\times\R)\, \qquad \mbox{up to zero $\H^3$-measure sets}.
\end{equation}
If we now pick a sequence $\{s_h\}_{h\in\N}$ such that $s_h\to 0$ as $h\to\infty$ and $\eta_{0,s_h}(E_1)\to\widetilde{E_1}$ in $L^1_{\rm loc}(\R^3)$, then, by \eqref{c15} and $u(0)=0$, we find that
\[
\Big\{(z,t):t\le \langle \nabla u_0(0), e_1\rangle \langle z , e_1\rangle \Big\}\subset \widetilde{E_1}\,,
\]
so that, thanks to \eqref{c14}, $\xi(\widetilde{E_1})>\xi(E_1)$. Since $\widetilde{E_1}\in TV(0,E_1)\subset TV(y,\Omega)$, this contradicts \eqref{c7}, and completes the proof of \eqref{ccccc}. Since $\partial \Omega\cap  B_{\rho} (y)\setminus An$ is smooth, by \eqref{c5}, \eqref{ccccc}, and \eqref{c8}, we conclude that there exists $\nu\in \mathbb S^{2}\cap \langle e_1,e_3\rangle\setminus\{\pm e_1\}$ such that
\begin{equation}
  \label{c16}
 \pa W \subset \pa E_1\,\qquad\mbox{and}\qquad W \subset E_1\,, \,\qquad\mbox{where}\qquad W:=\{\langle x , \nu \rangle \le0\}\cap H_1.
\end{equation}
The two properties above \eqref{c16} imply that 
\begin{equation}\label{c4}
\H^2(\partial (E_1\setminus W) \cap \partial W)=0,
\end{equation}
compare with \cite[Section 16.1]{maggiBOOK} or with \cite[Equation (2.7)]{DM}. We claim that this implies that $W=E_1$ up to zero $\H^3$-measure sets. Indeed if this was not case, by \eqref{c4} we would deduce that $\FF(\partial W)<\FF(\partial E_1)$, which would contradict \eqref{c3}, as $E_1 \Delta F \subset H$ by \eqref{c2}. Hence $W\in TV(y,\Omega)$, which implies the desired \eqref{c1}.
\end{proof}

\end{document}